\documentclass[11pt,letterpapert]{article}
\usepackage{graphicx}
\usepackage{amssymb}
\usepackage{epstopdf}
\usepackage{fancyhdr}
\usepackage{amsmath}
\usepackage{amsthm}
\usepackage{hyperref}
\usepackage{makeidx}
\usepackage{blkarray}
\usepackage{mathdesign}
\usepackage{color}
\usepackage{geometry}
\usepackage{soul}
\usepackage{nameref}
\usepackage{enumitem}

\definecolor{purple}{rgb}{.9,0,.9}

\newcommand\subsetsim{\mathrel{\substack{
  \textstyle\subset\\[-0.2ex]\textstyle\sim}}}

\makeatletter
\let\orgdescriptionlabel\descriptionlabel
\renewcommand*{\descriptionlabel}[1]{%
  \let\orglabel\label
  \let\label\@gobble
  \phantomsection
  \edef\@currentlabel{#1}%
  \let\label\orglabel
  \orgdescriptionlabel{#1}%
}
\makeatother

\graphicspath{{./figures/}}

\newcommand{\Nat}{\mathbb{N}}

\newcommand{\Real}{\mathbb{R}}

\newcommand{\rfa}{\quad {\rm for \ all}\ }

\newcommand{\cA}{{\cal A}}\newcommand{\cB}{{\cal B}}\newcommand{\cC}{{\cal C}}
\newcommand{\cD}{{\cal D}}\newcommand{\cE}{{\cal E}}\newcommand{\cF}{{\cal F}}
\newcommand{\cG}{{\cal G}}\newcommand{\cH}{{\cal H}}\newcommand{\cI}{{\cal I}}
\newcommand{\cJ}{{\cal J}}
\newcommand{\cM}{{\cal M}}\newcommand{\cN}{{\cal N}}\newcommand{\cO}{{\cal O}}
\newcommand{\cP}{{\cal P}}\newcommand{\cR}{{\cal R}}
\newcommand{\cS}{{\cal S}}\newcommand{\cT}{{\cal T}}\newcommand{\cU}{{\cal U}}
\newcommand{\cW}{{\cal W}}\newcommand{\cX}{{\cal X}}

\newcommand{\ba}{{\bf a}}\newcommand{\bb}{{\bf b}}\newcommand{\bc}{{\bf c}}
\newcommand{\bd}{{\bf d}}\newcommand{\be}{{\bf e}}
\newcommand{\bh}{{\bf h}}

\newcommand{\bm}{{\bf m}}\newcommand{\bn}{{\bf n}}

\newcommand{\bt}{{\bf t}}\newcommand{\bu}{{\bf u}}
\newcommand{\bv}{{\bf v}}\newcommand{\bw}{{\bf w}}
\newcommand{\bA}{{\bf A}}
\newcommand{\bB}{{\bf B}}\newcommand{\bC}{{\bf C}}\newcommand{\bD}{{\bf D}}

\newcommand{\bI}{{\bf I}}
\newcommand{\bL}{{\bf L}}

\newcommand{\bR}{{\bf R}}

\newcommand{\bnu}{\boldsymbol{\nu}}

\newcommand{\bchi}{\boldsymbol{\chi}}
\newcommand{\ve}{\varepsilon}

\newcommand{\bkappa}{\boldsymbol{\kappa}}
\newcommand{\btau}{\boldsymbol{\tau}}

\newtheorem{theorem}{Theorem}[section]
\newtheorem{lemma}[theorem]{Lemma}

\newtheorem{proposition}[theorem]{Proposition}

\newcommand\Tstrut{\rule{0pt}{2.6ex}}         

\newcommand{\beqn}{\begin{equation}}
\newcommand{\eeqn}{\end{equation}}

\newcommand{\bzero}{{\bf 0}}

\title{A fractional notion of length\\ and an associated nonlocal curvature\footnote{This is a corrected version of my paper \cite{S20} published under the same name.}}

\author{Brian Seguin}

\begin{document}
\date{}

\maketitle

\tableofcontents

\begin{abstract}
\noindent Here a new notion of fractional length of a smooth curve, which depends on a parameter $\sigma$, is introduced that is analogous to the fractional perimeter functional of open sets.  It is shown that in an appropriate limit the fractional length converges to the traditional notion of length up to a multiplicative constant.  Since a curve that connects two points of minimal length must have zero curvature, the Euler--Lagrange equation associated with the fractional length is used to motivate a nonlocal notion of curvature for a curve.  This is analogous to how the fractional perimeter has been used to define a nonlocal mean-curvature.\\

\noindent \emph{Dedicated to Eliot Fried, whose guidance following my time as a graduate student will always be appreciated.}

\end{abstract}

\section{Introduction}

\subsection{Background}

The origins of fractional perimeter and nonlocal curvature began with the work of Caffarelli, Roquejoffre, and Savin \cite{CRS10} who defined, up to a multiplicative constant, the $\sigma$-perimeter, for $0<\sigma<1$, of a measurable set $E\subseteq\Real^n$ relative to an open, bounded set $\Omega\subseteq\Real^n$ by
\beqn
\text{Per}_\sigma(E,\Omega):=\mathcal{I}(E\cap\Omega,E^c\cap\Omega)+\cI(E\cap\Omega,E^c\cap\Omega^c)+\cI(E\cap\Omega^c,E^c\cap\Omega),
\eeqn
where
\beqn
\cI(A,B):=\frac{1}{\alpha_{n-1}}\int_A\int_B |x-y|^{-n-\sigma}dxdy,\qquad A\cap B=\emptyset,
\eeqn
and $\alpha_{n-1}$ is the volume of the unit ball in $\Real^{n-1}$.  The first term in this definition is related to the fractional Sobolev space seminorm $|\chi_E|_{H^{\sigma/2}(\Omega)}$, and can be viewed as the fractional perimeter of $E$ inside of $\Omega$, while the other two terms can be interpreted as the fractional perimeter near $\partial\Omega$.  The study of functionals of this kind goes back to the work of Visintin \cite{V91}.  It is known \cite{CV11} that if the boundary of $E$ is smooth, then
\beqn\label{sPerlim}
\lim_{\sigma\uparrow 1} (1-\sigma)\text{Per}_\sigma(E,B_r)=\cH^{n-1}(\partial E\cap B_r)
\eeqn
for almost every $r>0$, where $B_r$ is the ball centered at the origin of radius $r$.  A set $E\subseteq\Real^n$ is a minimizer of the $\sigma$-perimeter relative to $\Omega$ if over all measurable sets $F\subseteq\Real^n$ such that $E\setminus \Omega=F\setminus \Omega$ we have
\beqn
\text{Per}_\sigma(E,\Omega)\leq \text{Per}_\sigma(F,\Omega).
\eeqn
Besides the relation \eqref{sPerlim}, it is known that the $\sigma$-perimeter functional $\Gamma$-converges to the classical notion of perimeter \cite{LPM11}.  

If the boundary of a minimizer $E$ is sufficiently regular, then it must satisfy
\beqn\label{PerEL}
\int_{\Real^n}\frac{\tilde\chi_E(x)}{|z-x|^{n+\sigma}} dx=0\qquad \text{for all}\ z\in\partial E,
\eeqn
where $\tilde\chi_E:=\chi_E-\chi_{E^c}$, $\chi_E$ is the characteristic function for the set $E$, and this integral is taken in the principle-value sense.  Because of the connection between the $\sigma$-perimeter and the areal measure \eqref{sPerlim}, and the fact that surfaces that minimize their area subject to a fixed boundary condition must have zero mean curvature, it is reasonable to define a nonlocal mean-curvature by
\beqn\label{NLMC}
H_\sigma(z):=\frac{1}{\omega_{n-2}}\int_{\Real^n}\frac{\tilde\chi_E(x)}{|z-x|^{n+\sigma}} dx\qquad \text{for all}\ z\in\partial E,
\eeqn
where $\omega_{n-2}$ is the $(n-2)$-dimensional measure of the unit sphere in $\Real^{n-1}$.  Notice that this quantity is independent of $\Omega$ and, hence, well-defined for any point on the surface that is the boundary of the set $E$.  Assuming that $\partial E$ is smooth, this curvature converges to the classical mean-curvature \cite{AV14} in the following sense:
\beqn\label{limHs}
\lim_{\sigma\uparrow 1}(1-\sigma)H_\sigma(z)=H(z).
\eeqn

The asymptotics of the fractional perimeter and nonlocal curvature as $\sigma$ goes to zero have also been studied.  Namely, it was shown in \cite{DF13} that
\beqn
\lim_{\sigma\downarrow 0} \sigma \text{Per}_\sigma(E,\Omega)=\frac{1}{\alpha_{n-1}}\big[(1-a(E))\cH^n(E\cap\Omega)+a(E)\cH^n(\Omega\backslash E)\big],
\eeqn
where $a(E):=\displaystyle \lim_{\sigma\downarrow 0}\frac{\sigma}{\omega_{n-1}} \int_{E\backslash B_1} |y|^{-n-\sigma}\, dy$, and in \cite{DV18} that
\beqn
\lim_{\sigma\downarrow 0} \sigma H_\sigma(z)=\frac{\omega_{n-1}}{\omega_{n-2}}.
\eeqn

The minimizers of the $\sigma$-perimeter functional, called $\sigma$-minimal surfaces, have been studied in great detail in recent years.  The regularity of $\sigma$-minimal surfaces has been investigated by Valdinoci and collaborators \cite{CL13,DV99,FV15??,SV13}. Among other things, it is known that $\sigma$-minimal surfaces are smooth off of a singular set of dimension at most $n-8$ for $\sigma$ sufficiently close to $1$. While this is in agreement with a well-known result for classical minimal surfaces \cite{G84}, $\sigma$-minimal surfaces may have features different from their classical counterparts, in that they may stick to the boundary of $\Omega$ \cite{DSV17,DV99}. The motion of surfaces by nonlocal mean-curvature has been investigated using level set methods \cite{CMP15,CMP12,CMP13,I09}.

\subsection{Extension and motivation}

The above discussion of nonlocal mean-curvature applies to surfaces that are the boundary of a set.  However,  Paroni, Podio-Guidugli, and Seguin discovered that it is possible to define these concepts for any smooth (hyper)surface \cite{PPGS99}.  The main idea is to define a fractional notion of area and find a condition similar to \eqref{PerEL} that a minimizer of this functional must satisfy.  Towards this end, they first showed that for a bounded set $E$ with smooth boundary and bounded, open $\Omega$ containing $E$ one can write
\beqn\label{Pareas}
\text{Per}_\sigma(E,\Omega)= \frac{1}{\alpha_{n-1}} \int_E\int_{E^c} |x-y|^{-n-\sigma}dxdy = \frac{1}{2\alpha_{n-1}}\int_{\cX(\partial E)}|x-y|^{-n-\sigma}dxdy,
\eeqn
where $\cX(\partial E)$ is the set of all pairs $(x,y)\in\Real^n\times\Real^n$ such that the oriented line segment connecting $x$ to $y$ crosses $\partial E$ an odd number of times.  The validity of \eqref{Pareas} follows from the fact that $\cX(\partial E)$ and $(E\times E^c)\cup (E^c\times E)$ agree up to a set of $\cH^{2n}$-measure zero.  As the far right-hand side of \eqref{Pareas} is expressed using $\partial E$, and not the set $E$, this motivates the following definition of the $\sigma$-area for a smooth surface $\cS$ with or without boundary:
\beqn\label{sArea}
\text{Area}_\sigma(\cS,\Omega):=\frac{1}{2\alpha_{n-1}}\int_{\cX(\cS)}|x-y|^{-n-\sigma}\max\{\chi_\Omega(x),\chi_\Omega(y)\}dxdy,
\eeqn
where it is assumed that $\cS$ is contained in $\Omega$.  The presence of $\max\{\chi_\Omega(x),\chi_\Omega(y)\}$ in the integrand is necessary to ensure the integral converges.  In this way, it is similar to the role $\Omega$ plays in the definition of the $\sigma$-perimeter.  It follows from \eqref{Pareas} that in the case where $\cS=\partial E$ and $E\subseteq \Omega$ that $\text{Area}_\sigma(\cS,\Omega)=\text{Per}_\sigma(E,\Omega)$.  The $\sigma$-area satisfies a limit relationship analogous to \eqref{sPerlim}.  It was shown \cite{PPGS99} that if $\cS$ minimizes the $\sigma$-area relative to all smooth, bounded, oriented surfaces in $\Omega$ that have the same boundary as $\cS$, then $\cS$ must satisfy
\beqn\label{sAreaEL}
\int_{\cA_e(z)}|z-y|^{-n-\sigma}dy-\int_{\cA_i(z)}|z-y|^{-n-\sigma}dy=0\rfa z\in \cS,
\eeqn
where
\begin{align}
\nonumber\cA_e(z)&:=\big\{y\in \Real^n\ |\ \big((z,y)\in \cX(\cS)\ \text{and}\ (z-y)\cdot \bn(z)> 0\big)\\
\label{Ae}&\hspace{1in}\text{or } \big((z,y)\in \cX(\cS)^c\ \text{and}\ (z-y)\cdot \bn(z)< 0\big)\big\},\\
\nonumber\cA_i(z)&:=\big\{y\in \Real^n\ |\ \big((z,y)\in \cX(\cS)^c\ \text{and}\ (z-y)\cdot \bn(z)> 0\big)\\
\label{Ai}&\hspace{1in}\text{or } \big((z,y)\in \cX(\cS)\ \text{and}\ (z-y)\cdot \bn(z)< 0\big)\big\}.
\end{align}
See Figure~\ref{AIAEcolor} for a depiction of these sets.

\begin{figure}[h]
\centering
\includegraphics[width=4in]{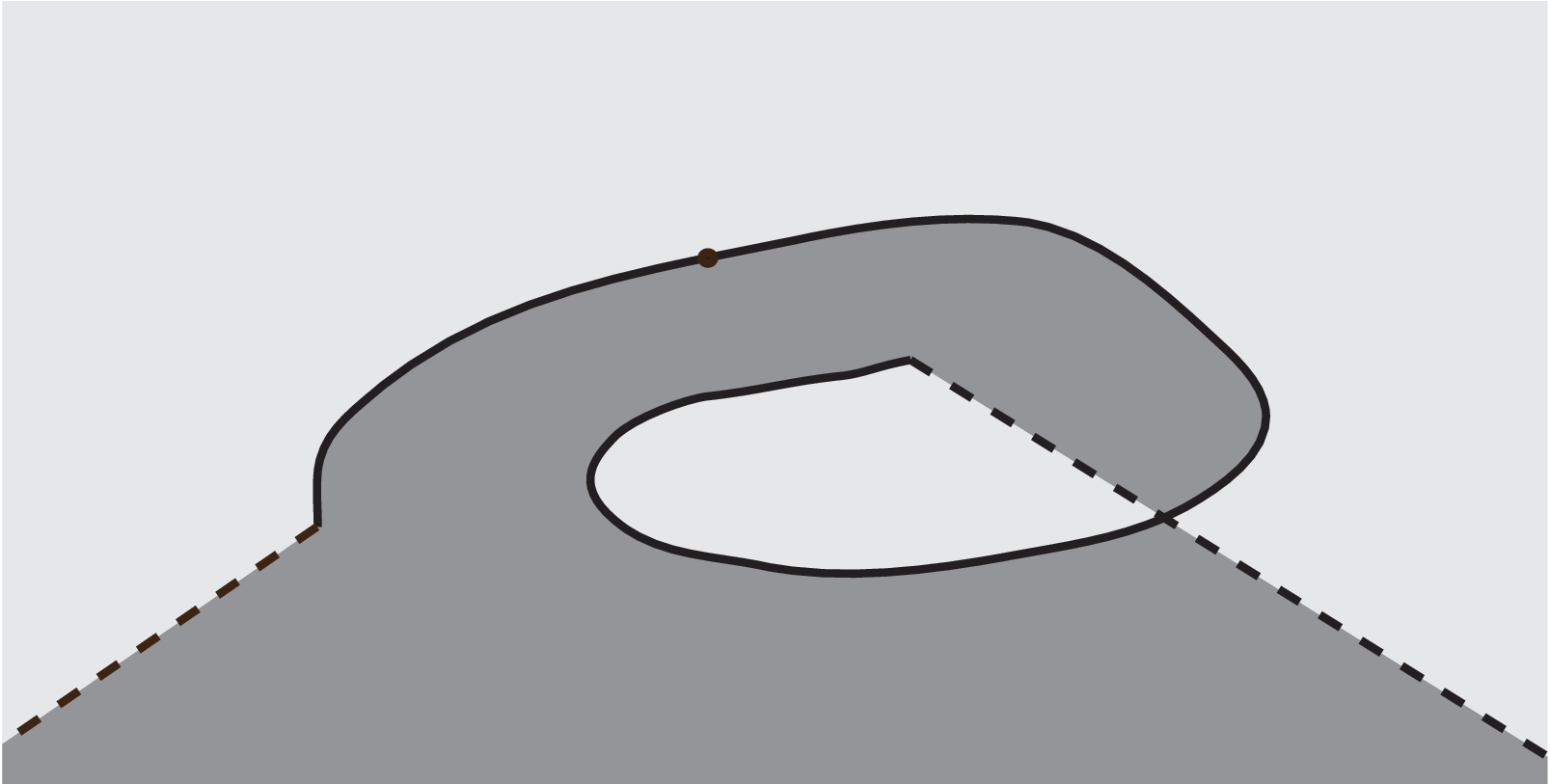}
\thicklines
\put(-170,100){$z$}
\put(-100,110){$\cS$}
\put(-250,100){$\cA_e(z)$}
\put(-150,20){$\cA_i(z)$}
\put(-158,110){$\bn(z)$}
\put(-175,110){\rotatebox[origin=c]{103}{$\vector(1,0){35}$}}
\caption{The solid line depicts $\cS$. The set of points of density 1 for $\cA_e(z)$ is shown in light grey, and the set of points of density 1 for $\cA_i(z)$ is in dark grey.  The dashed lines depict the part of the essential boundary between these sets that is not part of $\cS$.}
\label{AIAEcolor}
\end{figure}

This motivates defining the nonlocal mean-curvature of $\cS$ at $z$ using the opposite of the left-hand side of \eqref{sAreaEL}---that is,
\beqn\label{HsC}
H_\sigma(z):=\frac{1}{\omega_{n-2}} \int_{\Real^n} \frac{\hat\chi_\cS(z,y)}{|z-y|^{n+\sigma}}dy\rfa z\in\cS,
\eeqn
where
\beqn\label{chicS}
\hat\chi_\cS(z,y):=
\begin{cases}
1 & y\in\cA_i(z),\\
0 & y\not\in \cA_i(z)\cup\cA_e(z),\\
-1 & y\in\cA_e(z)
\end{cases}
\eeqn
and it is understood the integral is computed as a principle value.  The opposite of the left-hand side of \eqref{sAreaEL} is used so that the nonlocal mean-curvature of a sphere with outward orientation is negative, just as in the case for the classical mean-curvature.  Notice that $H_\sigma$ does not depend on $\Omega$.  Unsurprisingly, this curvature satisfies the limit relation \eqref{limHs}.  

To motivate a definition of fractional length, we will consider the $\sigma$-area in two dimensions, where a hypersurface is a curve.  When $n=2$, the $\sigma$-area becomes
\beqn\label{sArea2}
\text{Area}_\sigma(\cS,\Omega)=\frac{1}{4} \int_{\cX(\cS)}\frac{\max\{\chi_\Omega(x),\chi_\Omega(y)\}}{|x-y|^{2+\sigma}}dxdy.
\eeqn
The domain of integration here consists of line segments that are described by their endpoints.  A given line segment connecting $x$ to $y$ can be viewed as a one-dimensional disc and, hence, can be described by its midpoint $p$, a unit vector $\bu$ normal to the disc, and a radius $r$ so that 
\beqn
(x,y)=(p-r\bu',p+r\bu'),
\eeqn
where $\bu'$ is obtained by rotating $\bu$ clockwise by $90^\circ$.  Utilizing this change of variables, \eqref{sArea2} can be rewritten as
\beqn\label{FALmot}
\text{Area}_\sigma(\cS,\Omega)=\frac{1}{2} \int_{\cD(\cS)}(2r)^{-1-\sigma}\max\{\chi_\Omega(p-r\bu'),\chi_\Omega(p+r\bu')\}d\cH^4(p,\bu,r),
\eeqn
where $\cD(\cS)$ consists of all triples $(p,\bu,r)$ describing those one-dimensional discs that intersect $\cS$ an odd number of times and $\cH^4$ is the $4$-dimensional Hausdorff measure.  It is this formula for the fractional length that can be generalized to a curve in $n$ dimensions.  

Before this generalization is done, we first study the measure theoretic properties of the set of all discs that intersect a curve an odd number of times and other related sets of discs in Section~\ref{sectdiscs}.  In Section~\ref{sectFL} the fractional length is defined and it is shown that it converges, in an appropriate limit, to the classical notion of length up to a multiplicative constant.  Next, Section~\ref{sectVar} is dedicated to computing the Euler--Lagrange equation associated with the fractional length and this result is used to motivate a definition of nonlocal curvature for a curve.  The Appendix contains several change of variables formulas that are useful in established the desired results as well as a transport theorem that is applied to compute the first variation of the fractional length.


\section{Sets of discs}\label{sectdiscs}

In this section the set of all $(n-1)$-dimensional discs, and various subsets of it, are studied in $\Real^n$, with $n\geq 2$.  The results established here make precise which discs are integrated over in the definition of the nonlocal length.  Moreover, they will be crucial in computing the first variation of the fractional length.  We use $\cU_n$ to denote the set of unit vectors in $\Real^n$, and set
$$\cU_\perp^2:=\{(\ba,\bb)\in\cU_n\times\cU_n\ |\ \ba\cdot\bb=0\},$$
which consists of all pairs of orthogonal unit vectors.

The $(n-1)$-dimensional disc with center $p$, normal unit-vector $\bu$, and radius $r$ is denoted by
$$D(p,\bu,r):=\{ p+\xi\bv\ |\ (\bu,\bv)\in\cU_\perp^2,\ \xi\in[0,r)\}.$$
By the boundary $\partial D(p,\bu,r)$ of one of these discs we mean the $(n-2)$-dimensional manifold
\beqn
\{ p+r\bv\in\Real^n\ |\ \bv\in\cU_n\cap\{\bu\}^\perp\},
\eeqn
where $\{\bu\}^\perp$ is the set of all vectors orthogonal to $\bu$.  The disc together with its boundary is denoted by $\bar D(p,\bu,r)$.  Thus, the set of all discs in $\Real^n$ can be described by elements $(p,\bu,r)$ of the set $\cD:=\Real^n\times\cU_n\times\Real^+$, where $\Real^+:=(0,\infty)$.  For this reason, we will sometimes refer to the elements of $\cD$ as discs.

Consider a $C^1$ curve $\cC$ in $\Real^n$ whose closure $\bar\cC$ is a $C^1$, compact curve with boundary points $\partial \cC$ such that $\bar\cC=\cC\cup\partial\cC$. It is not assumed that $\cC$ is connected, so $\partial\cC$ could consist of any positive, even number of points.  Orient $\bar\cC$ so that at each point $z\in\bar\cC$ we have a unit tangent $\bt(z)$.  Consider the following subsets of the set of all discs $\cD$:
\begin{align*}
\cD_{\partial\cC1}&:=\{(p,\bu,r)\in\cD\ |\  \cH^0(\bar D(p,\bu,r)\cap \partial\cC)=1\}, \\
\cD_{\partial\cC2}&:=\{(p,\bu,r)\in\cD\ |\  \cH^0(\bar D(p,\bu,r)\cap \partial\cC)\geq 2\}, \\
\cD_{\partial\cC}&:= \cD_{\partial\cC 1}\cup \cD_{\partial\cC 2},\\
\cD_{\text{tan}}&:= \{ (p,\bu,r)\in\cD\ |\ \text{there is a } z\in \bar D(p,\bu,r)\cap \bar\cC\ \text{such that}\ \bt(z)\cdot \bu=0\},\\
\cD_\infty&:= \{(p,\bu,r)\in\cD\ |\ \cH^0(\bar D(p,\bu,r)\cap \bar\cC)=\infty\},\\
\cD_{\partial 1}&:= \{(p,\bu,r)\in \cD\ |\ \cH^0(\partial D(p,\bu,r)\cap \cC)=1\},\\
\cD_{\partial 2}&:= \{(p,\bu,r)\in \cD\ |\ \cH^0(\partial D(p,\bu,r)\cap \cC)\geq 2\},\\
\cD_{\partial}&:= \cD_{\partial 1}\cup \cD_{\partial 2},\\
\cD_\text{odd}&:=\{ (p,\bu,r)\in\cD\backslash (\cD_{\partial\cC}\cup\cD_{\rm tan}\cup \cD_\partial)\ |\ \cH^0(D(p,\bu,r)\cap \cC) \text{ is an odd number}\},\\
\cD_\text{even}&:=\{ (p,\bu,r)\in\cD\backslash (\cD_{\partial\cC}\cup\cD_{\rm tan}\cup \cD_\partial)\ |\ \cH^0(D(p,\bu,r)\cap \cC) \text{ is an even number}\}.
\end{align*}
The following lemma discusses the measure theoretic properties of these sets.  

\begin{lemma}\label{lemmeasure}
The following facts are true:
\begin{enumerate}
\item\label{Iep} $\cH^{2n-1}(\cD_{\partial\cC}\cap \cE)<\infty$ for any bounded, open set $\cE\subseteq \cD$,
\item\label{Itan} $\cH^{2n-1}(\cD_{\rm tan}\cap \cE)<\infty$ for any bounded, open set $\cE\subseteq \cD$,
\item\label{Ipartial} $\cH^{2n-1}(\cD_\partial\cap \cE)<\infty$ for any bounded, open set $\cE\subseteq \cD$,
\item\label{Iinf} $\cD_\infty\subseteq\cD_\text{\rm tan}$,
\item\label{Ipartial2} $\cH^{2n-1}(\cD_{\partial 2})=\cH^{2n-1}(\cD_{\partial\cC 2})=0$,
\item\label{IH2n10} $\cH^{2n-1}(\cD_{\partial\cC}\cap\cD_{\partial})=0$,
\item\label{Ieo} $\cD_\text{\rm even}$ and $\cD_\text{\rm odd}$ are open subsets of $\cD$,
\item\label{Idun} $\cD=\cD_\text{\rm odd}\cup\cD_\text{\rm even}\cup \cD_{\partial\cC}\cup \cD_{\rm tan}\cup \cD_\partial$.
\end{enumerate}
\end{lemma}

\begin{proof} Let $\cE$ be a bounded, open subset of $\cD$.  Find $R>0$ such that if $(p,\bu,r)\in\cE$, then $r\in (0,R]$.  Set $\cE_\Xi=\Xi^{-1}(\cE)$ and $\cE_{\Psi}=\Psi^{-1}(\cE)$, where $\Xi$ and $\Psi$ are defined in \eqref{XiCOV} and \eqref{PsiCOV}, respectively.

Item \ref{Iep})  Consider the set 
\beqn\label{ApC}
\cA_{\partial\cC}:=\partial\cC\times\cU_\perp^2\times\Real^+_0\times\Real^+
\eeqn
and the function $\Xi:\cA_{\partial\cC}\rightarrow\cD$ defined in \eqref{XiCOV} of the Appendix.  Notice that $\cD_{\partial\cC}\cap \cE \subseteq \Xi(\cA_{\partial\cC}\cap\cE_\Xi)$.  Since $\Xi$ is Lipschitz on $\cA_{\partial\cC}\cap\cE_\Xi$ and $\cH^{2n-1}(\cA_{\partial\cC}\cap\cE_\Xi)<\infty$, it follows that $\cH^{2n-1}(\cD_{\partial\cC}\cap\cE)<\infty$.

Item \ref{Itan}) The proof is the same as Item~\ref{Iep} with the exception that one uses the set
\beqn\label{Atan}
\cA_\text{tan}:=\bigcup_{z\in\bar\cC} \{z\} \times \{(\ba,\bb)\in\cU^2_\perp\ |\ \bb\cdot\bt(z)=0\} \times \Real^+_0\times\Real^+
\eeqn
rather than $\cA_{\partial\cC}$.

Item \ref{Ipartial})  Consider the set 
\beqn\label{Ap}
\cA_\partial:=\cC\times\cU_\perp^2\times \Real^+
\eeqn
and the function $\Psi:\cA_\partial\rightarrow\cD$ defined by \eqref{PsiCOV} in the Appendix.  Notice that $\cD_\partial\cap \cE \subseteq \Psi(\cA_\partial\cap\cE_\Psi)$.  Since $\Psi$ is Lipschitz on $\cA_\partial\cap\cE_\Psi$ and $\cH^{2n-1}(\cA_\partial\cap\cE_\Psi)<\infty$, it follows that $\cH^{2n-1}(\cD_\partial\cap\cE)<\infty$.

Item \ref{Iinf}) Consider $(p,\bu,r)\in \cD_\infty$, so that there are an infinite number of points in $\bar D(p,\bu,r)\cap\bar\cC$.  Since this set is compact it follows that this intersection has a cluster point, say $z\in \bar D(p,\bu,r) \cap\bar\cC$.  Suppose that $(p,\bu,r)\not \in \cD_\text{tan}$, so that $\bt(z)\cdot \bu\not = 0$.  Since in a neighborhood of $z$ the curve $\bar\cC$ can be approximated by its tangent line which has direction $\bt(z)$, it follows that there are no points in this neighborhood besides $z$ in the intersection $\bar D(p,\bu,r)\cap\bar\cC$.  This contradicts the fact that $z$ is a cluster point of $\bar D(p,\bu,r) \cap\bar\cC$.  Thus, we must have $z\in\cD_\text{tan}$.  

Item \ref{Ipartial2}) Consider the set
\beqn
\cA_{\partial 2}:=\{(z_1,z_2,\ba,\bb)\in \cC\times\cC\times \cU^2_\perp \ |\ \ba\cdot (z_2-z_1)>0,\, \bb\cdot(z_2-z_1)=0\}.
\eeqn
and define the function $\Lambda:\cA_{\partial 2}\rightarrow \cD$ by
\beqn
\Lambda(z_1,z_2,\ba,\bb):=(z_1+\frac{|z_2-z_1|^2}{2\ba\cdot(z_2-z_1)}\ba,\bb,\frac{|z_2-z_1|^2}{2\ba\cdot(z_2-z_1)}).
\eeqn
One can check that the boundary of the disc $D(\Lambda(z_1,z_2,\ba,\bb))$ intersects $\cC$ at $z_1$ and $z_2$.  Thus, $\cD_{\partial 2}\subseteq \Lambda(\cA_{\partial 2})$.  Moreover, $\Lambda$ is locally Lipschitz on $\cA_{\partial 2}$.  It follows that since $\cH^{2n-1}(\cA_{\partial 2})=0$, we must have $\cH^{2n-1}(\cD_{\partial 2})=0$.

Now consider the set
\beqn\label{ApC2}
\cA_{\partial\cC2}:=\{(q_1,q_2,\ba,\bb,r)\in\partial\cC\times\partial\cC\times\cU_\perp^2\times\Real^+\ |\  \bb\cdot(q_1-q_2)=0\}.
\eeqn
Notice that $\cD_{\partial\cC2} \subseteq \Xi(\cA_{\partial\cC2})$.  Since $\Xi$ is locally Lipschitz on $\cA_{\partial\cC}$ and $\cH^{2n-1}(\cA_{\partial\cC2})=0$, it follows that $\cH^{2n-1}(\cD_{\partial\cC2})=0$.

Item \ref{IH2n10}) Consider the set
\beqn
\cA_{\partial\cC\partial}:=\{(q,z,\ba,\bb,r)\in\partial\cC\times\cA_\partial\ |\ (q-z)\cdot\bb=0\}
\eeqn
and the function $\Psi_\partial:\cA_{\partial\cC\partial}\rightarrow\cD$ defined by
\beqn
\Psi_\partial(q,z,\ba,\bb,r):=\Psi(z,\ba,\bb,r)\quad \text{for all}\ (q,z,\ba,\bb,r)\in\cA_{\partial\cC\partial}.
\eeqn
Notice that $\cD_{\partial\cC}\cap\cD_\partial\subseteq \Psi_\partial(\cA_{\partial\cC\partial})$ and $\cH^{2n-1}(\cA_{\partial\cC\partial})=0$.  It follows that $\cH^{2n-1}(\cD_{\partial\cC}\cap\cD_\partial)=0$.

Item \ref{Ieo}) This follows from Item~\ref{Iinf} and the definitions of $\cD_\text{even}$ and $\cD_\text{odd}$.

Item \ref{Idun}) This follows from Item~\ref{Iinf} and the definitions of the various sets involved.
\end{proof}

The previous result yields enough information to obtain the properties of $\cD_\text{odd}$ we require. To state the desired result, it is useful to introduce the following notation: if $A$ and $B$ are subsets of $\cD$, write
\beqn\label{subsetsim}
A\subsetsim B\quad\text{if}\quad \cH^{2n-1}(A\setminus B)=0
\eeqn
and
\beqn\label{cong}
A\cong B\quad \text{if}\quad A\subsetsim B\ \text{and}\ B\subsetsim A.
\eeqn
We now argue that $\cD_{\rm odd}$ is a locally of finite perimeter and classify is essential boundary.

\begin{proposition}\label{propFP}
The set $\cD_\text{\rm odd}$ is locally of finite perimeter.  Moreover, the essential boundary\footnote{For the definition of sets of finite perimeter and essential boundary see, for example, Ambrosio, Fusco, and Pallara \cite{AFP}.} $\partial^*\cD_\text{\rm odd}$ of this set satisfies $\partial^*\cD_\text{odd}\cong\cD_{\partial\cC1}\cup\cD_{\partial 1}$.
\end{proposition}

\begin{proof}
From Items \ref{Ieo} and \ref{Idun} of Lemma~\ref{lemmeasure}, we see that $\partial^* \cD_\text{odd}\subseteq \cD_{\partial\cC}\cup\cD_\text{tan}\cup \cD_\partial$.  Thus, from Items \ref{Iep}--\ref{Ipartial} of the same lemma, whenever $\cE\subseteq\cD$ is a bounded, open set we have
\beqn
\cH^{2n-1}(\partial^*\cD_\text{odd}\cap \cE)\leq \cH^{2n-1}((\cD_{\partial\cC}\cup\cD_\text{tan}\cup \cD_\partial)\cap\cE)<\infty.
\eeqn
By a result of Federer, see 4.5.11 of \cite{Fed}, we can conclude that $\cD_\text{odd}$ has finite perimeter in $\cE$.  Moreover, it is known, see Ambrosio, Fusco, and Pallara \cite{AFP} Theorem~3.61, that it follows that $\cD_\text{odd}$ has density either $0$, $1/2$, or $1$ at $\cH^{2n-1}$-a.e.~point of $\cE$, and $\partial^*\cD_\text{odd}\cap \cE$ consists of those points with density $1/2$ up to a set of $\cH^{2n-1}$-measure zero.  

We first show that $\cD_\text{tan}\setminus (\cD_{\partial\cC}\cup\cD_\partial)$ has density either 0 or 1 relative to $\cD_\text{odd}$ at $\cH^{2n-1}$-a.e.~point and, hence, cannot be part of $\partial^*\cD_\text{odd}$.  Fix $(p,\bu,r)\in\cD_\text{tan}\setminus (\cD_{\partial\cC}\cup\cD_\partial)$.  Find a small, connected neighborhood $\cN$ of $(p,\bu,r)$ in $\cD$ that is disjoint from $\cD_\partial$ and $\cD_{\partial\cC}$, which is possible since these two sets are closed in $\cD$.  From Item~\ref{Itan} of Lemma~\ref{lemmeasure}, for $\cH^{2n}$-a.e.~$(p',\bu',r')\in\cN$, the disc $D(p',\bu',r')$ is not tangent to $\cC$ and from Item~\ref{Iinf} such discs only intersect $\cC$ a finite number of times.  Consider two such discs $(p_1,\bu_1,r_1),(p_2,\bu_2,r_2)\in\cN$.  Since $\cN$ is a small neighborhood, we know that the discs $D(p_1,\bu_1,r_1)$ and $D(p_2,\bu_2,r_2)$ are close in the sense that their centers are close, their orientations are close, and their radii are close.  Since $\cN$ is connected, there is a continuous path in $\cN$ from $(p_1,\bu_1,r_1)$ to $(p_2,\bu_2,r_2)$.  In the process of going along this path, the disc $D(p_1,\bu_1,r_1)$ sweeps out a tube in $\Real^n$ until it reaches $D(p_2,\bu_2,r_2)$.  Let $\cT$ denote this tube.  The boundary of $\cT$ consists of the two discs $D(p_1,\bu_1,r_1)$  and $D(p_2,\bu_2,r_2)$ along with the side of the tube, which is obtained by starting with $\partial D(p_1,\bu_1,r_1)$ and proceeding along the path of discs until one reaches $\partial D(p_2,\bu_2,r_2)$.  Let $\cS$ denote the side of the tube.  Since $\cN$ is disjoint from $\cD_\partial$ and $\cD_{\partial\cC}$, it follows that 
\beqn\label{tubneigh}
\cT\cap\partial\cC=\emptyset\qquad\text{and}\qquad\cS\cap\cC=\emptyset.
\eeqn

We proceed by showing that $\cH^0(D(p_1,\bu_1,r_1)\cap\cC)$ and $\cH^0(D(p_2,\bu_2,r_2)\cap\cC)$ have the same parity, meaning that they are either both even or both odd, by showing that their sum is even. If both of these numbers are zero, we are done.  Let $\bar\cC_c$ denote one of the finitely many connected components of $\cC$. Now suppose that after starting at one of the endpoints of $\bar\cC_{c}$ and going along this curve, it intersects one of the discs.  As the two discs are part of the boundary of the tube $\cT$, this means that after this intersection the curve $\bar\cC_c$ has either entered or left the tube.  The latter option is not possible since this would mean that the curve would have entered the tube previously through the side $\cS$, which would violate \eqref{tubneigh}$_2$. After the curve intersects one of the discs and is inside $\cT$, it cannot end inside $\cT$ as this would violate \eqref{tubneigh}$_1$.  Thus, the curve must leave $\cT$.  This can be accomplished by either crossing one of the two discs or by leaving $\cT$ through the side $\cS$ of the tube.  However, this last option is not possible as it violates \eqref{tubneigh}$_2$. Thus, the curve $\cC_c$ must intersect one of the discs to exit $\cT$.  After the curve leaves the tube $\cT$, this argument can be repeated again and again until the other endpoint of $\cC_c$ is reached.  Regardless of how many times this is repeated, the total number of intersections the curve $\cC_c$ has with $D(p_1,\bu_1,r_1)$ and $D(p_2,\bu_2,r_2)$ is even. Repeating this argument for each connected component of $\cC$ shows that $\cH^0(D(p_1,\bu_1,r_1)\cap\cC)$ and $\cH^0(D(p_2,\bu_2,r_2)\cap\cC)$ have the same parity.  If the parity is even, then almost all discs in $\cN$ are also in $\cD_\text{even}$, and so the density of $\cD_\text{odd}$ at $(p,\bu,r)$ is 0, while if the parity is odd, then it is 1.

We conclude that $\partial^*\cD_\text{odd}\cap \cE\subsetsim(\cD_{\partial\cC}\cup \cD_\partial)\cap\cE$.  It then follows from Item \ref{Ipartial2} of the lemma that $\partial^*\cD_\text{odd}\cap \cE\subsetsim(\cD_{\partial\cC1}\cup \cD_{\partial 1})\cap\cE$. 

Next we establish that $(\cD_{\partial 1}\setminus\cD_{\partial\cC})\cap\cE\subseteq \partial^*\cD_\text{odd}\cap \cE$ by arguing that all points $(p,\bu,r)\in(\cD_{\partial 1}\setminus\cD_{\partial\cC})\cap\cE$ have density $1/2$ relative to $\cD_\text{odd}$.  First notice that $\cD_\partial=\Psi(\cA_\partial)$, where $\cA_\partial$ is defined in \eqref{Ap} and $\Psi$ is the function defined in \eqref{PsiCOV} of the Appendix.  This means that $\cD_\partial$ is an immersed submanifold of $\cD$.  Moreover, the function $\Psi$ is an embedding on the preimage of $\cD_{\partial 1}$ under $\Psi$ and, so, $\cD_{\partial 1}$ is a $(2n-1)$-dimensional embedded submanifold of $\cD$.  It follows from the definitions of the sets involved that if $(p,\bu,r)\in \cD_{\partial 1}\setminus\cD_{\partial\cC}$, then any neighborhood of $(p,\bu,r)$ in $\cD$ contains elements of $\cD_\text{odd}$ and $\cD_\text{even}$.  Putting this together with the fact that $\cD_{\partial 1}$ is a $(2n-1)$-dimensional embedded submanifold of $\cD$ and $\cD_{\partial\cC}$ is a closed set, we can conclude that the density at $(p,\bu,r)$ of $\cD_\text{odd}$ must be 1/2.  

The proof that $(\cD_{\partial\cC}\setminus\cD_{\partial1})\cap\cE\subseteq \partial^*\cD_\text{odd}\cap \cE$ uses a similar argument and, so, will be skipped.  It then follows from Items \ref{Ipartial2} and \ref{IH2n10} of Lemma \ref{lemmeasure} that $(\cD_{\partial\cC1}\cup\cD_{\partial1})\cap\cE\subseteq \partial^*\cD_\text{odd}\cap \cE$.
\end{proof}

Since $\cD_\text{odd}$ is locally a set of finite perimeter, it has an exterior unit normal at $\cH^{2n-1}$-a.e.~point of its essential boundary.  The next result describes this normal vector along the part of $\partial^*\cD_\text{odd}$ that we will need later.

\begin{proposition}\label{redbdy}
For $\cH^{2n-1}$-a.e.~$(p,\bu,r)\in\partial^*\cD_\text{\rm odd}$ such that $D(p,\bu,r)\cap\partial\cC$ is empty, there is a unique $z\in \partial D(p,\bu,r)\cap\cC$.  Moreover, for such $(p,\bu,r)$ the exterior unit-normal $\bnu(p,\bu,r)\in \Real^n\times \{\bu\}^\perp\times\Real$ is given by
\beqn
\bnu(p,\bu,r):=
\begin{cases}
\bm(p,\bu,r) & \text{if } \cH^0(D(p,\bu,r)\cap \cC \text{ is an odd number},\\
-\bm(p,\bu,r) &\text{if } \cH^0(D(p,\bu,r)\cap \cC \text{ is an even number},
\end{cases}
\eeqn
where
\beqn\label{normalm}
\bm(p,\bu,r):=\frac{\Big( z-p+\frac{(p-z)\cdot\bt}{\bu\cdot\bt}\bu, \frac{(p-z)\cdot\bt}{\bu\cdot\bt}(p-z),r\Big)}{\sqrt{ |p-z|^2 + \Big(\frac{(p-z)\cdot\bt}{\bu\cdot\bt}\Big)^2|z-p|^2 + \Big(\frac{(p-z)\cdot\bt}{\bu\cdot\bt}\Big)^2 + r^2}}
\eeqn
and $\bt$ is a tangent to $\cC$ at $z$.
\end{proposition}

\begin{proof}
We are looking for the exterior unit-normal to $\cD_\text{odd}$ on $\partial^*\cD_\text{odd}\setminus \cD_{\partial\cC}$.  By Proposition~\ref{propFP} and Item~\ref{IH2n10} of Lemma~\ref{lemmeasure} it suffices to find the exterior unit-normal on
\beqn
\cP:=\partial^*\cD_\text{odd}\cap (\cD_{\partial 1}\setminus \cD_{\partial\cC}).
\eeqn
As argued in Proposition~\ref{propFP}, $\cD_{\partial 1}$ is an embedded submanifold of $\cD$ and, thus, is $\cH^{2n-1}$-rectifiable.  It follows that the approximate tangent space to $\cP$, where it exists, coincides with the tangent space of $\cD_{\partial 1}$.  Thus, to calculate the exterior unit-normal on $\cP$, we first find the tangent space to $\cD_{\partial 1}$ at any point.

Let $(p,\bu,r)\in \cP$ and find the unique $(z,\ba,\bb,r)\in\cA_\partial$, see \eqref{Ap}, that gets mapped to $(p,\bu,r)$ under $\Psi$, which is defined in \eqref{PsiCOV}.  Denote by $\bt$ a unit tangent vector to $\cC$ at $z$.  Since $(p,\bu,r)\not \in \cD_\text{tan}$, we must have $\bb\cdot\bt\not =0$.  A curve in $\cA_\partial$ that passes through $(z,\ba,\bb,r)$ induces, via the mapping $\Psi$, a curve in $\cD_{\partial 1}$ that passes through $(p,\bu,r)$.  By differentiating such curves we can generate vectors in the tangent space of $\cD_{\partial 1}$ at $(p,\bu,r)$.  In particular, one can find that the following vectors are in the tangent space:
\beqn\label{TSbasis}
(\bt,\bzero,0),\quad (\bc,\bzero,0),\quad (\bzero,\bd,0),\quad (\ba,0,1),\quad (r\bb,-\ba,0),
\eeqn
where $\bc$ and $\bd$ are any vectors orthogonal to both $\ba$ and $\bb$.  This generates a list of $2n-1$ linearly independent vectors since $\bb\cdot\bt\not = 0$.  Thus, these vectors span the tangent space at $(p,\bu,r)$.  A vector in $\Real^n\times\{\bu\}^\perp\times\Real$, which is the tangent space to $\cD$ at $(p,\bu,r)$, that is orthogonal to the list of vectors in \eqref{TSbasis} is
\beqn
\big(-\ba+\tfrac{\ba\cdot\bt}{\bb\cdot\bt}\bb,r\tfrac{\ba\cdot\bt}{\bb\cdot\bt}\bb,1\big).
\eeqn
Since $(p,\bu,r)=\Psi(z,\ba,\bb,r)=(z+r\ba,\bb,r)$, we can replace $\ba$ with $(p-z)/r$ and $\bb$ with $\bu$.  Doing so and normalizing this vector results in the vector $\bm$ defined in \eqref{normalm}.  

The vector $\bm$ at $(p,\bu,r)$ is pointing outward from $\cD_\text{odd}$ if the interior of the disc associated with $(p,\bu,r)$ crosses $\cC$ an odd number of times.  To see this, let $\gamma$ be a smooth curve in $\cD$ defined on an interval of $\Real$ containing zero such that $\gamma(0)=(p,\bu,r)$ and $\gamma'(0)=\bm$.  For small negative values of $t$, we have $\gamma(t) \in \cD_\text{odd}$ since the last component of $\bm$ is positive and, so, the disc $D(\gamma(t))$ will intersect the curve an odd number of times and its boundary will not intersect the curve.  Moreover, $\gamma(t)\not\in \cD_\text{odd}$ for small positive $t$ because the curve $\cC$ will cross the disc $D(\gamma(t))$ one more time than the disc $D(p,\bu,r)$ since the boundary of this disc intersects the curve.  Using similar logic, one can see that $-\bm$ is the outward normal if $D(p,\bu,r)\cap \cC$ has an even number of points.
\end{proof}

We need to analyze one final set of discs which will be of use later. To describe it, we make use of the following notation: if $\bu$ is a vector, then let $P_\bu$ denote the orthogonal projection onto the direction $\bu$ and $P_\bu^\perp$ the orthogonal projection onto the subspace perpendicular to $\bu$.

\begin{lemma}\label{lemDR}
Given a $z\in\Real^n$, $\bt\in\cU_n$, $K>0$, and $\alpha\in(0,1)$, define the region
\beqn\label{defcR}
\cR:=\{z+s\bt+\bh\ |\ s\in\Real,\ \bh\in\Real^n,\ \bh\cdot\bt=0,\ \text{\rm and}\ |\bh|<K|s|^{1+\alpha}\}.
\eeqn
If
\beqn\label{defcDcD}
\cD_\cR:=\{(\ba,\bb,r)\in \cU^2_\perp\times (0,\infty)\ |\ \bb\cdot\bt>0\ \text{\rm and}\ D(z+r\ba,\bb,r)\cap \cR\not=\emptyset\},
\eeqn
then it is true that
\beqn\label{cDRinc}
\cD_\cR\subseteq\Big\{(\ba,\bb,r)\in \cU^2_\perp\times (0,\infty)\ \Big|\  r\geq\frac{|\bb\cdot\bt|^{1/\alpha}}{2K^{1/\alpha}}\Big\}.
\eeqn
\end{lemma}

\begin{proof}
Fix $(\ba,\bb,r)\in\cD_\cR$. First notice that this means that $|\bb\cdot\bt|\not=1$, as otherwise the disc $D(z+r\ba,\bb,r)$ would not intersect $\cR$. Consider the minimization problem
\beqn
\inf |p-z|^2,\qquad p\in \cR\cap D(z+r\ba,\bb,r).
\eeqn
Find $p\in \bar\cR\cap\bar D(z+r\ba,\bb,r)$ that achieves this infimum. Since $\cR$ is open and $\bar D(z+r\ba,\bb,r)$ is convex, it is not possible for $p\in \cR\cap \bar D(z+r\ba,\bb,r)$ and, so, we must have $p\in \partial\cR\cap \bar D(z+r\ba,\bb,r)$. The set $\partial\cR$ consists of those $q\in\Real^n$ satisfying
\beqn
|P_\bt^\perp(q-z)|^2=K^2|P_\bt(q-z)|^{2+2\alpha},
\eeqn
and $\bar D(z+r\ba,\bb,r)$ consists of those points $q$ that lie in the plane $\bb\cdot(q-z)=0$ and satisfy $|q-z-r\ba|^2\leq r^2$. It follows that there are Lagrange multipliers $\lambda_1$ and $\lambda_2$ and a KKT (Karush--Kuhn--Tucker) multiplier $\mu$ such that $p$ satisfies
\begin{align*}
2(p-z)&=2\lambda_1\big([p-z-[(p-z)\cdot\bt]\bt-(1+\alpha)K^2[(p-z)\cdot\bt]^{1+2\alpha}\bt\big)+\lambda_2\bb+2\mu(p-z-r\ba),\\
0&=|P_\bt^\perp(p-z)|^2-K^2|P_\bt(p-z)|^{2+2\alpha},\\
0&=\bb\cdot(p-z),\\
0&=(|p-z-r\ba|^2- r^2)\mu,\\
0&\geq |p-z-r\ba|^2-r^2.
\end{align*}
The first of the above equations implies that $p-z$ is in the subspace spanned by $\bb$, $\bt$, and $\ba$. Let $p'$ denote the projection of $p$ onto the two-dimensional plane containing $z$ and spanned by $\bb$ and $\bt$. Notice that $|p'-z|\leq|p-z|$ and $p'\in\cR$.

\begin{figure}[h]
\centering\thicklines
\hspace{-.75in}
\includegraphics[width=3in]{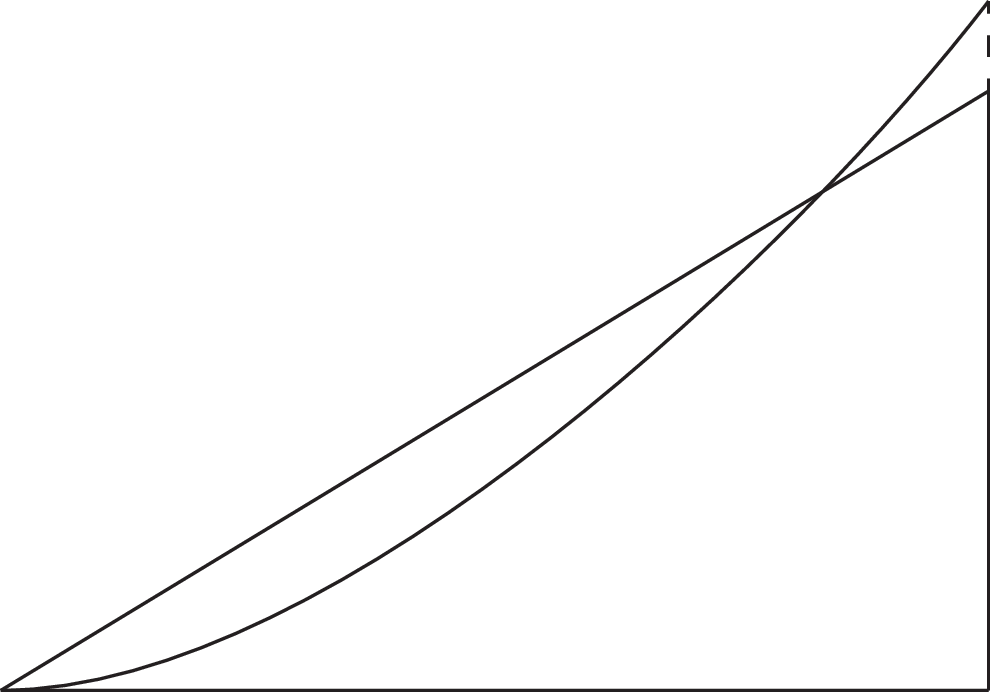}
\put(-227,-2){$z$}
\put(-219,-2){$\bullet$}
\put(-13,-13){$z+s\bt$}
\put(-3,-2){$\bullet$}
\put(0,0.5){\rotatebox[origin=c]{0}{$\vector(1,0){25}$}}
\put(20,5){$\bt$}
\put(5,129){$p'$}
\put(-3,129){$\bullet$}
\put(-150,80){$|p'-z|$}
\put(5,60){$|p'-z-s\bt|\leq K|s|^{1+\alpha}$}
\put(-110,-10){$s$}
\put(-217.5,-10){\rotatebox[origin=c]{-53}{$\vector(1,0){25}$}}
\put(-198,-15){$\bb$}
\qbezier(-199,10.8)(-190,10)(-190,1)
\put(-185,10){$\theta$}
\put(5,150){$y$}
\put(-3,148){$\bullet$}
\caption{A depiction of the geometry in the plane containing $z$ and spanned by $\bt$ and $\bb$. The straight line connecting $z$ and $p'$ is contained in the plane orthogonal to $\bb$, while the curved arc connecting $z$ and $y$ is part of $\partial\cR$. Notice that $\cos\theta=\sqrt{1-(\bb\cdot\bt)^2}$, from which \eqref{planegeo}$_1$ follows.}
\label{figplane}
\end{figure}

Analyzing the geometry within the plane containing $z$ and spanned by $\bb$ and $\bt$, see Figure~\ref{figplane}, one finds that 
\beqn\label{planegeo}
|p'-z|=\frac{|s|}{\sqrt{1-(\bb\cdot\bt)^2}}\quad \text{and}\quad K|s|^{1+\alpha}\geq |p'-z-s\bt|=\frac{|s||\bb\cdot\bt|}{\sqrt{1-(\bb\cdot\bt)^2}}.
\eeqn
The second of these identities implies that
\beqn\label{abss}
|s|\geq\Big(\frac{|\bb\cdot\bt|}{K\sqrt{1-(\bb\cdot\bt)^2}}\Big)^{1/\alpha}.
\eeqn
Since $p\in \bar D(z+r\ba,\bb,r)$, it follows from \eqref{planegeo}$_1$ and \eqref{abss} that
\beqn\label{DRrel1}
2r\geq|p-z|\geq |p'-z|\geq\frac{|\bb\cdot\bt|^{1/\alpha}}{K^{1/\alpha}\sqrt{1-(\bb\cdot\bt)^2}^{1+1/\alpha}}\geq \frac{|\bb\cdot\bt|^{1/\alpha}}{K^{1/\alpha}}.
\eeqn
It follows that $(\ba,\bb,r)$ is contained in the right-hand side of \eqref{cDRinc}.
\end{proof}

\begin{lemma}\label{lemIB}
Given $z\in\Real^n$, $\bt\in\cU_n$, and $\alpha>\sigma$, define $\cR$ and $\cD_\cR$ as in Lemma~\ref{lemDR}. If
\beqn\label{cDcReta}
\cD_\cR(\eta):=\{(\ba,\bb,r)\in\cD_\cR\ |\ r<\eta\},
\eeqn
then for sufficiently small $\eta>0$ there is a constant $C$ depending only on $K$, $\sigma$, $\alpha$, and $n$ such that
\beqn\label{intRbound}
\int_{\cD_\cR(\eta)}r^{-1-\sigma}d\cH^{2n-2}(\ba,\bb,r)\leq C\eta^{\alpha-\sigma}.
\eeqn
\end{lemma}

\begin{proof}
Set $\cU_n(\eta):=\{\bb\in\cU_n\ |\ |\bb\cdot\bt|\leq K (2\eta)^\alpha\}$. Using Lemma~\ref{lemDR} and the coarea formula, the integral in \eqref{intRbound} can be written using iterated integrals as
\begin{align}
\int_{\cD_\cR(\eta)}r^{-1-\sigma}d\cH^{2n-2}(\ba,\bb,r)&\leq\sqrt{2}\int_{\cU_n(\eta)}\int_{\cU_n\cap\{\bb\}^\perp}\int_{\frac{|\bb\cdot\bt|^{1/\alpha}}{2K^{1/\alpha}}}^\eta r^{-1-\sigma}drd\ba d\bb\nonumber\\
&=\frac{\sqrt{2}\omega_{n-2}}{\sigma}\int_{\cU_n(\eta)}\Big[\frac{2^\sigma K^{\sigma/\alpha}}{|\bb\cdot\bt|^{\sigma/\alpha}}-\eta^{-\sigma}\Big] d\bb.
\end{align}
To continue the calculation, we apply the coarea formula again, this time representing those vectors in $\cU_n(\eta)$ using $\bt$, a unit vector orthogonal to $\bt$, and angle (see Mihaila and Seguin \cite[Lemma A.2]{MS24}) to find that
\begin{align}
&\nonumber\int_{\cD_\cR(\eta)}r^{-1-\sigma}d\cH^{2n-2}(\ba,\bb,r)\\
&\hspace{.2in}\nonumber\leq\frac{\sqrt{2}\omega_{n-2}}{\sigma}\int_{\{\bt,-\bt\}} \int_{\cU_n\cap\{\bt\}^\perp}\int_{\arccos(K(2\eta)^\alpha)}^{\pi/2} \Big[ 2^\sigma K^{\sigma/\alpha}\cos^{-\sigma/\alpha}\theta-\eta^{-\sigma}\Big]\sin^{n-2}\theta d\theta d\cH^{n-2} d\cH^0\\
&\hspace{.2in}=\frac{2\sqrt{2}\omega_{n-2}^2}{\sigma}\int_{\arccos(K(2\eta)^\alpha)}^{\pi/2} \Big[ 2^\sigma K^{\sigma/\alpha}\cos^{-\sigma/\alpha}\theta-\eta^{-\sigma}\Big]\sin^{n-2}\theta d\theta.
\end{align}
Now use the change of variables $x=\sin^2\theta$ to find that
\begin{multline}\label{bddcov}
\int_{\cD_\cR(\eta)}r^{-1-\sigma}d\cH^{2n-2}(\ba,\bb,r)\\
\leq\frac{\sqrt{2}\omega_{n-2}^2}{\sigma}\int_{1-K^2(2\eta)^{2\alpha}}^1[2^\sigma K^{\sigma/\alpha}(1-x)^{-\frac{1}{2}-\frac{\sigma}{2\alpha}}-\eta^{-\sigma}(1-x)^{-1/2}x^\frac{n-3}{2}]dx.
\end{multline}
We will obtain bounds on each of the two terms on the right-hand side of the previous inequality separately. For the first term, the integral with respect to $x$ can be calculated exactly. Once this is done, the fact that $\eta^{\frac{1}{2}-\frac{\sigma}{2\alpha}}\leq \eta^{\alpha-\sigma}$ for small $\eta$ shows that
\beqn\label{bddcov1}
\int_{1-K^2(2\eta)^{2\alpha}}^1(1-x)^{-\frac{1}{2}-\frac{\sigma}{2\alpha}} dx\leq \frac{2^{1+\frac{\alpha-\sigma}{2\alpha}}  \alpha K^{\alpha-\sigma}}{\alpha-\sigma}\eta^{\alpha-\sigma}.
\eeqn
To bound the second term in \eqref{bddcov}, first since $n\geq 2$ and $1-K^2(2\eta)^{2\alpha}\leq x\leq 1$, it follows that for sufficiently small $\eta$,
\beqn
x^{\frac{n-3}{2}}\leq x^{-1/2}\leq (1-K^2(2\eta)^{2\alpha})^{-1/2}\leq 2.
\eeqn
Upon using this fact, the integral in the second term of the right-hand side of \eqref{bddcov} can be computed to obtain
\beqn\label{bddcov2}
\int_{1-K^2(2\eta)^{2\alpha}}^1\eta^{-\sigma}(1-x)^{-1/2}x^\frac{n-3}{2}dx\leq 4\sqrt{2}K\eta^{\alpha-\sigma}.
\eeqn
Upon using \eqref{bddcov1} and \eqref{bddcov2} in \eqref{bddcov}, one obtains the desired bound \eqref{intRbound}.
\end{proof}

\section{Fractional length}\label{sectFL}

In this section we define a fractional notion of length and show that in an appropriate limit as $\sigma$ goes to 1, this converges to the $\cH^1$ measure up to a multiplicative constant.

Let $\Omega$ be an open, bounded set that contains $\bar\cC$.  Motivated by \eqref{FALmot}, given the curve $\cC$, define the $\sigma$-length of $\cC$ relative to $\Omega$ by
\beqn\label{defLens1}
\text{Len}_\sigma(\cC,\Omega):=\int_{\cD(\cC)} r^{1-n-\sigma}\sup_{\ba\in\cU_n\cap\{\bu\}^\perp} \chi_\Omega(p+r\ba) d\cH^{2n}(p,\bu,r),
\eeqn
where $\cD(\cC):=\cD_\text{odd}$. Using the definition
\beqn\label{defDO}
\cD_\Omega(\cC):=\{(p,\bu,r)\in \cD(\cC)\ |\ \sup_{\ba\in\cU_n\cap\{\bu\}^\perp} \chi_\Omega(p+r\ba)=1\},
\eeqn
the fractional length can be rewritten as
\beqn\label{defLens2}
\text{Len}_\sigma(\cC,\Omega)=\int_{\cD_\Omega(\cC)} r^{1-n-\sigma} d\cH^{2n}(p,\bu,r).
\eeqn

To show that this definition yields a finite number, first notice that
\beqn\label{Psiinc}
\cD_\Omega(\cC)\subseteq \Xi \Big(\bigcup_{\xi\in\Real^+} \cC\times \cU_\perp^2\times\{\xi\}\times[\xi,\xi + d(\Omega)]\Big),
\eeqn
where $d(\Omega)$ is the diameter of $\Omega$ and $\Xi$ is defined in \eqref{XiCOV} of the Appendix.  To see this, consider $(p,\bu,r)\in\cD_\Omega(\cC)$.  Since $D(p,\bu,r)$ intersects $\cC$ a finite number of times, we can find $z\in\cC\cap D(p,\bu,r)$ with minimum distance to $p$ such that $(p-z)\cdot \bu=0$.  Set $\xi=|p-z|$, $\ba=(p-z)/\xi$, and $\bb=\bu$.  It follows that $\Xi(z,\ba,\bb,\xi,r)=(p,\bu,r)$.  Since $z$ is the closest point on $\cC$ to $p$ in $D(p,\bu,r)$, we must have $\xi\leq r$ otherwise $D(p,\bu,r)$ would not intersect $\cC$.  Moreover, $r\leq \xi+d(\Omega)$ since if this were not true then $\partial D(p,\bu,r)\cap \Omega=\emptyset$.  It follows that $(p,\bu,r)$ is an element of the set on the right-hand side of \eqref{Psiinc}, so \eqref{Psiinc} holds. Thus, we can utilize the change of variables formula \eqref{XiCOV2} in the Appendix to find that
\begin{align*}
\text{Len}_\sigma(\cC,\Omega)&\leq \int_{\cC}\int_0^\infty\int_{\cU_\perp^2}\int_\xi^{\xi+d(\Omega)} 2^{-1/2}r^{1-n-\sigma}\xi^{n-2} |\bb\cdot \bt(z)|drd\cH^{2n-3}(\ba,\bb) d\xi dz\\
&\leq \cH^1(\cC) \cH^{2n-3}(\cU_\perp^2) \int_0^\infty\int_\xi^{\xi+d(\Omega)} r^{1-n-\sigma}\xi^{n-2} drd\xi\\
&=\frac{1}{2-n-\sigma}\cH^1(\cC) \cH^{2n-3}(\cU_\perp^2) \int_0^\infty\Big( \frac{\xi^{n-2}}{(\xi+d(\Omega))^{n-2+\sigma}}-\xi^{-\sigma}\Big) d\xi,
\end{align*}
and the remaining integral involving $\xi$ is finite.

The next goal is to show that the fractional length converges in an appropriate limit to the classical notion of length up to some multiplicative constant.  Doing so will require the following result.

\begin{lemma}\label{lemint}
For any $\bc\in\cU_n$, we have
\beqn
\int_{\cU^2_\perp} |\bb\cdot \bc| d\cH^{2n-3}(\ba,\bb)=\frac{4\sqrt{2}\pi^{n-1}}{\Gamma(\tfrac{n+1}{2})\Gamma(\tfrac{n-1}{2})}
\eeqn
where $\Gamma$ is the gamma function.
\end{lemma}

\begin{proof}
First notice that for any vector $\bv\in\Real^{n-1}$, by the area formula we have
\beqn
\int_{\cU_{n-1}} |\bb\cdot \bv| d\bb=2\int_{\cU_{n-2}}\int_0^{\pi/2} |\bv|\cos\theta(\sin\theta)^{n-3} d\theta d\cH^{n-2}=\frac{\omega_{n-3}\Gamma(\tfrac{n-2}{2})|\bv|}{\Gamma(\tfrac{n}{2})}=\frac{2\pi^{\tfrac{n-2}{2}}|\bv|}{\Gamma(\tfrac{n}{2})},
\eeqn
where
\beqn
\omega_{n-3}=\cH^{n-3}(\cU_{n-2})=\frac{2\pi^{\tfrac{n-2}{2}}}{\Gamma(\tfrac{n-2}{2})}.
\eeqn
Letting $P_\ba$ denote the projection onto the plane orthogonal to $\ba$, we can compute using the coarea and area formulas that
\begin{align*}
\int_{\cU^2_\perp} |\bb\cdot \bc| d\cH^{2n-3}(\ba,\bb)&=\int_{\cU_n}\int_{\cU_n\cap\{\ba\}^\perp} \sqrt{2}|\bb\cdot P_\ba \bc| d\bb d\ba\\
& = \frac{2\sqrt{2}\pi^{\tfrac{n-2}{2}}}{\Gamma(\tfrac{n}{2})} \int_{\cU_n} |P_\ba\bc| d\ba\\
& = \frac{4\sqrt{2}\pi^{\tfrac{n-2}{2}}}{\Gamma(\tfrac{n}{2})} \int_{\cU_{n-1}}\int_0^{\pi/2} \sin\theta (\sin\theta)^{n-2} d\theta d\cH^{n-1}\\
& = \frac{4\sqrt{2}\pi^{\tfrac{n-2}{2}}}{\Gamma(\tfrac{n}{2})} \frac{\omega_{n-2}\Gamma(\tfrac{n}{2})\Gamma(\tfrac{1}{2})}{2\Gamma(\tfrac{n+1}{2})}\\
&=\frac{4\sqrt{2}\pi^{n-1}}{\Gamma(\tfrac{n+1}{2})\Gamma(\tfrac{n-1}{2})}.
\end{align*}
\end{proof}

\begin{theorem}\label{Thmlim}
If $\Omega\subseteq\Real^n$ is any open, bounded set such that $\bar\cC\subseteq\Omega$, then
\beqn\label{limresult}
\lim_{\sigma\uparrow 1} (1-\sigma){\rm Len}_\sigma(\cC,\Omega)=\frac{4\pi^{n-1}}{\Gamma(\tfrac{n+1}{2})\Gamma(\tfrac{n-1}{2})(n-1)}\cH^1(\bar\cC).
\eeqn
\end{theorem}

\begin{proof}
Begin by setting $\ve:=(1-\sigma)^{1/n}$ and
$$\cD_\ve(\cC):=\{(p,\bu,r)\in\cD(\cC)\ |\ r\leq \ve\}.$$
One can show that
\beqn
\cD_\Omega(\cC)\setminus \cD_\ve(\cC)\subseteq \Xi \Big(\bigcup_{\xi\in\Real^+} \cC\times \cU_\perp^2\times\{\xi\}\times[\max\{\xi,\ve\},\xi + d(\Omega)]\Big),
\eeqn
using an argument similar to that justifying \eqref{Psiinc}.  Thus, using the change of variables \eqref{XiCOV2} there is a constant $C_n$ depending on $\cC$ and $n$ such that
\begin{align*}
\int_{\cD_\Omega(\cC)\setminus \cD_\ve(\cC)} r^{1-n-\sigma}d\cH^{2n}(p,\bu,r)& \leq C_n \int_0^\infty\int_{\max\{\xi,\ve\}}^{\xi + d(\Omega)} \xi^{n-2} r^{1-n-\sigma} dr d\xi \\
&=\frac{C_n}{n+\sigma-2} \Big [ \int_0^\ve \xi^{n-2}[\ve^{2-n-\sigma}-(\xi+d(\Omega))^{n-2-\sigma}] d\xi\\
&\qquad+ \int_\ve^\infty [\xi^{-\sigma}-\xi^{n-2}(\xi+d(\Omega))^{2-n-\sigma}]d\xi\Big].
\end{align*}
Since
\beqn
\int_0^\ve \xi^{n-2}[\ve^{2-n-\sigma}-(\xi+d(\Omega))^{n-2-\sigma}] d\xi\leq \frac{\ve^{1-\sigma}}{n-1}\\
\eeqn
and
\beqn
\int_\ve^\infty [\xi^{-\sigma}-\xi^{n-2}(\xi+d(\Omega))^{2-n-\sigma}]d\xi\leq \frac{d(\Omega)(n+\sigma-2)\ve^{-\sigma}}{\sigma},
\eeqn
it follows that
\beqn
\lim_{\sigma\uparrow 1} (1-\sigma)\int_{\cD_\Omega(\cC)\setminus \cD_\ve(\cC)} r^{1-n-\sigma}d\cH^{2n}(p,\bu,r) = 0.
\eeqn
Thus,
\beqn\label{lim0}
\lim_{\sigma\uparrow 1} (1-\sigma)\text{Len}_\sigma(\cC,\Omega)=\lim_{\sigma\uparrow 1}\int_{\cD_\ve(\cC)}(1-\sigma)r^{1-n-\sigma} d\cH^{2n}(p,\bu,r).
\eeqn

Each $(p,\bu,r)\in\cD_\ve(\cC)$ may intersect $\cC$ multiple times, however we know it intersects $\cC$ at least once.  Thus, we can arbitrarily associate each $(p,\bu,r)$ with some point $z\in \cC$.  Let $c(p,\bu,r)$ denote the selected point in $\cC$. We can think of $c$ as a mapping from $\cD_\ve(\cC)$ to $\cC$.  Many such mappings exist, but here we select one.  For $z\in\cC$ and $(\ba,\bb)\in\cU_\perp^2$ set
\beqn
C_\ve(z,\ba,\bb):=\{(\xi,r)\in\Real^+\times\Real^+\ |\ (z+\xi\ba,\bb,r)\in\cD_\ve(\cC)\ \text{and}\ c(z+\xi\ba,\bb,r)=z\}.
\eeqn
It follows from the definition of $c$ that the function $\Xi$ defined in \eqref{XiCOV} is injective on the set
\beqn
\bigcup_{(z,\ba,\bb)\in \cC\times \cU_\perp^2} \{z\}\times\{\ba\}\times\{\bb\}\times C_\ve(z,\ba,\bb).
\eeqn
Thus, by the change of variables \eqref{XiCOV2} we have
\begin{multline}\label{lim1}
\int_{\cD_\ve(\cC)} r^{1-n-\sigma}d\cH^{2n}(p,\bu,r) \\= \int_\cC\int_{\cU_\perp^2} \int_{C_\ve(z,\ba,\bb)} 2^{-1/2} r^{1-n-\sigma} \xi^{n-2} |\bb\cdot \bt(z)| d\cH^2(\xi,r) d\cH^{2n-3}(\ba,\bb) dz.
\end{multline}
Since $\cC$ is a $C^1$ curve, for all $(z,\ba,\bb)$ such that $\ba$ is not parallel to $\bt(z)$ there is a $\ve_0$ such that if $\ve\leq\ve_0$ we have
\beqn
C_\ve(z,\ba,\bb)=\{(\xi,r)\in\Real^+\times\Real^+\ |\ \xi\in[0,\ve]\ \text{and}\ r\in[\xi,\ve]\}.
\eeqn
Thus,
\begin{align*}
\lim_{\sigma\uparrow 1}(1-\sigma) \int_{C_\ve(z,\ba,\bb)} r^{1-n-\sigma} \xi^{n-2} d\cH^2(\xi,r) & = \lim_{\sigma\uparrow 1}(1-\sigma) \int_{0}^\ve\int_\xi^\ve r^{1-n-\sigma} \xi^{n-2} drd\xi\\
&= \frac{1}{n-1}.
\end{align*}
Putting this together with \eqref{lim0} and \eqref{lim1} we find
\beqn
\lim_{\sigma\uparrow 1} (1-\sigma)\text{Len}_\sigma(\cC,\Omega) = \frac{2^{-1/2}}{n-1} \int_\cC\int_{\cU_\perp^2} |\bb\cdot \bt(z)| d\cH^{2n-3}(\ba,\bb) dz.
\eeqn
With the help of Lemma~\ref{lemint}, we obtain \eqref{limresult}.
\end{proof}


\section{Variation of Len$_\sigma$ and nonlocal curvature}\label{sectVar}

This section is dedicated to computing the Euler--Lagrange equation associated with the functional Len$_\sigma$ and using this to define a fractional notion of curvature. To compute this, we will use what is known as a transport theorem.  The version of this transport theorem applicable here can be found in the Appendix. 

\begin{theorem}
Assume that $\cC$ has $C^{1,\alpha}$, $\alpha>\sigma$, regularity. Choose an orientation for $\cC$ and let $\bt(z)$ be the unit tangent to $\cC$ at $z$ associated with this orientation. A necessary and sufficient condition for the vanishing of the first variation of $\text{Len}_\sigma(\cC,\Omega)$ with respect to curves with the same boundary as $\cC$ is that for all $z\in\cC$,\footnote{Here $\bb\wedge\ba$ is the linear mapping defined by $(\bb\wedge\ba)\bu:=(\ba\cdot\bu)\bb-(\bb\cdot\bu)\ba$ for all $u\in\Real^n$.}
\beqn\label{LenEL}
\lim_{\ve\rightarrow 0}\Big(\int_{\cA_\text{\rm o}^+(z,\ve)}-\int_{\cA_\text{\rm e}^+(z,\ve)}\Big) \frac{(\bb\wedge\ba)\bt(z)}{r^{1+\sigma}}d\cH^{2n-2}(\ba,\bb,r)=\bzero,
\eeqn
where 
\begin{align}
\cA_\text{\rm o}^+(z,\ve)&:=\{(\ba,\bb,r)\in\cU_\perp^2\times (\ve,\infty)\ |\ \cH^0(D(z+r\ba,\bb,r)\cap \cC) \text{ is odd}, \bb\cdot\bt(z)>0\},\\
\cA_\text{\rm e}^+(z,\ve)&:=\{(\ba,\bb,r)\in\cU_\perp^2\times (\ve,\infty)\ |\ \cH^0(D(z+r\ba,\bb,r)\cap \cC) \text{ is even}, \bb\cdot\bt(z)>0\}.
\end{align}
\end{theorem}

\begin{proof}
Fix $\eta>0$, and define
\beqn
f_\eta(p,\bu,r):=g_\eta(r)\sup_{\ba\in\cU_n\cap\{\bu\}^\perp} \chi_\Omega(p+r\ba)\quad \text{for all}\ (p,\bu,r)\in\cD,
\eeqn
where
\beqn
g_\eta(r):=
\begin{cases}
0 & \text{if}\ r\leq \eta,\\
r^{1-n-\sigma} & \text{if}\ r>\eta,
\end{cases}
\eeqn
so that $f_\eta$ is bounded. Furthermore, define the `truncated' $\sigma$-length by
\beqn
\text{Len}_{\sigma,\eta}(\cC,\Omega):=\int_{\cD(\cC)} f_\eta\, d\cH^{2n}.
\eeqn
To establish the desired result, we will compute the first variation of this truncated $\sigma$-length, and then send $\eta$ to zero.

Fix $z_\circ \in \cC$ and consider $\bw\in C^{\infty}_c(\Real^n,\Real^n)$ such that $\bw(z_\circ)\not=\bzero$ and $\text{supp}(\bw)$ does not intersect $\partial\cC$. For each $t\in \Real$ define the set
\beqn
\cC_t:=\{z+t\bw(z)\ |\ z\in\cC\}.
\eeqn
Due to the regularity of $\cC$, there is a sufficiently small interval $\cI$ containing zero such that for $t\in\cI$, $\cC_t$ is a $C^{1,\alpha}$ curve contained in $\Omega$. For each $t\in \cI$ and $z\in\cC_t$, there is a parameterization of $\cC_t$ near $z$ of the form
\beqn\label{phitz}
\phi_{t,z}(s)=z+s\bt(z)+\bh_{t,z}(s),\qquad s\in \cJ_{t,z},
\eeqn
where $\cJ_{t,z}$ is an interval containing zero and $\bh_{t,z}:\cJ_{t,z}\rightarrow\Real^n$ satisfies 
\beqn\label{fzprop}
\bh_{t,z}(0)=\textbf{0},\quad \bh_{t,z}'(0)=\textbf{0},\quad \text{and}\quad \bh_{t,z}(s)\cdot \bt(z)=0\ \text{for all } s\in\cJ_{t,z}.
\eeqn
As $\cC_t$ is $C^{1,\alpha}$, so are the functions $\bh_{t,z}$. The interval $\cI$ can be chosen small enough so that (i) there is a single interval $\cJ$ containing zero such that $\cJ\subseteq\cJ_{t,z}$ for all $t\in \cI$ and $z\in \text{supp}(\bw)$ and (ii) the H\"older constants for $\bh_{t,z}'$ can be chosen uniformly in $t\in \cI$ and $z\in \text{supp}(\bw)$. Thus, by \eqref{fzprop}$_{1,2}$ and the characterization of $C^{1,\alpha}$ functions by Anderson \cite{A97}, there is a $K>0$ such that 
\beqn\label{uniformK}
|\bh_{t,z}(s)|\leq K|s|^{1+\alpha},\qquad  s\in\cJ,\ t\in \cI.
\eeqn
For $z\in\cC_t$, let $\cR(z)$ be the set defined in \eqref{defcR} with the dependence on $z$ being made explicit, $\bt$ replaced by $\bt(z)$, and $K$ being as in \eqref{uniformK}. The fact that $\phi_{t,z}$ is a local parameterization for $\cC_t$ near $z$, \eqref{fzprop}$_3$ and \eqref{uniformK} imply that there is a $R>0$ such that if $B(z,R)$ is the open ball of radius $R$ centered at $z$, then
\beqn\label{cCcR}
B(z,R)\cap\cC_t\subseteq \overline{\cR(z)}\quad\text{for any}\ z\in\cC_t\cap \text{supp}(\bw),\ t\in\cI.
\eeqn

In an effort to apply Theorem~\ref{thmATTmod} to compute the derivative of $\text{Len}_{\sigma,\eta}(\cC_t,\Omega)$ with respect to $t$, let $\cN$ be the $(2n-1)$-dimensional manifold defined as the disjoint union of $\cA_{\partial \cC}$ and $\cA_{\partial}$ (see \eqref{ApC} and \eqref{Ap}), and define the function $\Theta:\cI\times\cN\rightarrow\cD$, by 
\beqn
\Theta(t,m):=
\begin{cases}
\Xi(m) & \text{if}\  m\in \cA_{\partial\cC},\\
\Psi(z+t\bw(z),\ba,\bb,r) & \text{if} \ m=(z,\ba,\bb,r)\in \cA_\partial,
\end{cases}
\eeqn
where $\Xi$ and $\Psi$ are defined in \eqref{XiCOV} and \eqref{PsiCOV}, respectively. We will use the notation $\Theta_t:=\Theta(t,\cdot)$ and a prime will denote a partial derivative with respect to $t$. Since $\Xi$ and $\Psi$ are $C^1$, so is $\Theta$. It follows from Items~\ref{Ipartial2} and \ref{IH2n10} of Lemma~\ref{lemmeasure} and Proposition~\ref{propFP} applied to the curve $\cC_t$ that $\Theta_t(\cN)\cong\partial^*\cD_\Omega(\cC_t)\cong \cD_{\partial\cC1}(t)\cup \cD_{\partial 1}(t)$ (see the notation introduced in \eqref{subsetsim} and \eqref{cong}), where 
\begin{align}
\cD_{\partial\cC1}(t)&:=\{(p,\bu,r)\in\cD\ |\  \cH^0(\bar D(p,\bu,r)\cap \partial\cC_t)=1\},\\
\cD_{\partial 1}(t)&:= \{(p,\bu,r)\in \cD\ |\ \cH^0(\partial D(p,\bu,r)\cap \cC_t)=1\}.
\end{align}
It also follows from Items~\ref{Ipartial2} and \ref{IH2n10} of Lemma~\ref{lemmeasure} that
\beqn
\cH^{2n-1}(\{ (p,\bu,r)\in\partial^*\cD(\cC_t)\ |\ \cH^0(\Theta_t^{-1}(\{(p,\bu,r)\})>1)\})=0.
\eeqn
One can check that the gradient of the functions $\Xi$ and $\Psi$ are injective $\cH^{2n-1}$-a.e. It follows that the gradient of $\Theta_t$ is also injective $\cH^{2n-1}$-a.e.~for sufficiently small $t$. Assume that $\cI$ is chosen small enough so that the gradient of $\Theta_t$ is $\cH^{2n-1}$-a.e.~injective for all $t\in\cI$. Thus, $\Theta$ satisfies the three conditions \ref{D1}--\ref{D3} with $\cO_t$ replaced by $\cD(\cC_t)$. Using the notation in \eqref{defcF} with $\cO_t$ replaced by $\cD(\cC_t)$, the velocity $\bv$ associated with $\Theta$, see \eqref{veldef}, is given by
\beqn\label{varvel}
\bv(t,p,\bu,r)=
\begin{cases}
(\textbf{0},\textbf{0},0) & \text{if}\  (p,\bu,r)\in \cD_{\partial \cC1}(t),\\
(\bw(z),\textbf{0},0) & \text{if} \ (p,\bu,r)\in \cD_{\partial 1}(t),
\end{cases}
\quad (t,p,\bu,r)\in \partial^*\cF,
\eeqn
where $z\in\cC_t$ is the unique point in $\partial D(p,\bu,r)\cap\cC_t$ when $(p,\bu,r)\in\cD_{\partial 1}(t)$. 

As already noted, $f_\eta$ is bounded and $f_\eta\in L^1(\cD(\cC_t))$ since $\text{Len}_{\sigma,\eta}(\cC_t,\Omega)$ is finite for all $t\in\cI$ by the argument at the beginning of Section~\ref{sectFL}. As $f_\eta$ is continuous at $\cH^{2n-1}$-a.e.~$(p,\bu,r)$, it follows that \ref{F2} holds. For any bounded interval $\cG\subseteq\cI$, one can use Proposition~\ref{redbdy} and the change of variables formula \eqref{PsiCOV2} in the Appendix to find that
\begin{align}
\int_\cG\int_{\partial^*\cO_t}|f_\eta\bv\cdot\bnu|\, d\cH^{2n-1}dt&=\int_\cG\int_{\cD_{\partial 1}(t)} f_\eta|\bv\cdot\bnu|\,d\cH^{2n-1}dt\\
&\leq \int_\cG \int_{\cC_t} \int_{\cU^2_\perp\times(\eta,\infty)} \frac{|\bw(z)|}{r^{1+\sigma}}\, d\cH^{2n-2}(\ba,\bb,r)dzdt.
\end{align}
This last integral is finite since $\bw$ is smooth with compact support and $\cG$ is bounded. Finally, using the area formula and the properties of $\Theta_t$, with $J(\Theta_t)$ denoting the Jacobian of $\Theta_t$, notice
\begin{align}
\int_{\partial^*\cD(\cC_t)}f_\eta\bv\cdot\bnu\, d\cH^{2n-1}&=\int_{\Theta_t(\cA_\partial)}f_\eta\bv\cdot\bnu\, d\cH^{2n-1}\\
&=\int_{\cA_\partial} (f_\eta\circ\Theta_t) (\Theta'_t\cdot\bnu\circ\Theta_t)J(\Theta_t)\, d\cH^{2n-1}\\
&=\int_\cC\int_{\cU^2_\perp}\int_\eta^\infty \frac{(\Theta'_t\cdot\bnu\circ\Theta_t)J(\Theta_t)}{r^{1+n+\sigma}}\, drd\cH^{2n-3}dz.
\end{align}
It follows from the regularity of $\Theta$ that the above expression is continuous in $t$.
As clearly $f_\eta'=0$, it has been shown that conditions \ref{F1}--\ref{F5} hold.

We may now apply Theorem~\ref{thmATTmod} to find that
\beqn\label{FVeta}
\big(\text{Len}_{\sigma,\eta}(\cC_t,\Omega)\big)'=\int_{\partial^* \cD(\cC_t)} f_\eta \bv\cdot\bnu d\cH^{2n-1}.
\eeqn
To rewrite the right-hand side of the previous equation, first use Proposition~\ref{redbdy} to obtain an expression for $\bnu$ and then the change of variables formula \eqref{PsiCOV2} in the Appendix. The resulting expression involves an integral over a subset of $\cC_t\times\cU^2_\perp\times (0,\infty)$, which is symmetric under the transformation $\bb\mapsto-\bb$. As the resulting integrand is even under this transformation, only integrating over those $(\ba,\bb)\in\cU^2_\perp$ such that $\bb\cdot\bt(z)>0$ and doubling the result yields
\beqn\label{varsLeneta}
\big(\text{Len}_{\sigma,\eta}(\cC_t,\Omega)\big)'=\int_{\cC_t}\int_{\cU^2_\perp}\int_\eta^\infty \frac{\sqrt{2}\zeta(t,z,\ba,\bb,r) \bw(z)\cdot(\bb\wedge\ba)\bt(z)}{r^{1+\sigma}}dr d\cH^{2n-3}(\ba,\bb)dz,
\eeqn
where
\beqn\label{defchi}
\zeta(t,z,\ba,\bb,r):=\begin{cases}
1 & \text{if}\ \cH^0(D(z+r\ba,\bb,r)\cap \cC_t) \text{ is odd and } \bb\cdot\bt(z)>0,\\
-1 & \text{if}\ \cH^0(D(z+r\ba,\bb,r)\cap \cC_t) \text{ is even and } \bb\cdot\bt(z)>0,\\
0 & \text{otherwise}.
\end{cases}
\eeqn
The goal now is to show that \eqref{varsLeneta}, viewed as a function of $t$, converges uniformly in $t$ as $\eta\rightarrow 0$ to
\beqn\label{limhdef}
h(t):=\lim_{\ve\rightarrow 0}\int_{\cC_t}\int_{\cU^2_\perp}\int_\ve^\infty \frac{\sqrt{2}\zeta(t,z,\ba,\bb,r) \bw(z)\cdot(\bb\wedge\ba)\bt(z)}{r^{1+\sigma}}dr d\cH^{2n-3}(\ba,\bb)dz.
\eeqn

First it must be argued that the limit used in defining $h(t)$ exists. Toward that end, let $h_\ve(t)$ be defined as the right-hand side of the previous equation without the limit in $\ve$ being taken. Notice that for $\ve>\ve'>0$,
\beqn\label{hvet}
|h_\ve(t)-h_{\ve'}(t)|=\Big|\int_{\cC_t}\int_{\cU^2_\perp}\int_{\ve'}^\ve \frac{\sqrt{2}\zeta(t,z,\ba,\bb,r) \bw(z)\cdot(\bb\wedge\ba)\bt(z)}{r^{1+\sigma}}dr d\cH^{2n-3}(\ba,\bb)dz\Big|.
\eeqn
For $z\in\cC_t$, consider the transformation 
\beqn
(\ba,\bb)\mapsto (P_{\bt(z)}\ba-P_{\bt(z)}^\perp\ba,P_{\bt(z)}\bb-P_{\bt(z)}^\perp\bb)=:(\tilde\ba,\tilde\bb),
\eeqn
which is a bijection on $\cU^2_\perp$. It follows from \eqref{cCcR} that for sufficiently small $\ve$ it is true that
\beqn
(\ba,\bb,r)\not\in\cD_{\cR(z)}(\ve)\Longleftrightarrow (\tilde\ba,\tilde\bb,r)\not\in\cD_{\cR(z)}(\ve),\qquad z\in \cC_t\cap\text{supp}(\bw).
\eeqn
This means that for $(\ba,\bb,r)\not\in\cD_{\cR(z)}(\ve)$, it is true that
\begin{align}
\zeta(t,z,\ba,\bb,r)&=\zeta(t,z,\tilde\ba,\tilde\bb,r)=-1\\
(\bb\wedge\ba)\bt(z)&=-(\tilde\bb\wedge\tilde\ba)\bt(z),
\end{align}
and, thus, outside of $\cD_{\cR(z)}(\ve)$ the integrand of the integral on the right-hand side of \eqref{hvet} is odd.
Thus, we conclude with the aid of Lemma~\ref{lemIB} that
\begin{align}
\nonumber |h_\ve(t)-h_{\ve'}(t)|&=\Big|\int_{\cC_t}\int_{\cD_{\cR(z)}(\ve)\setminus \cD_{\cR(z)}(\ve')} \frac{\sqrt{2}\zeta(t,z,\ba,\bb,r)\bw(z)\cdot(\bb\wedge\ba)\bt(z)}{r^{1+\sigma}} drd\cH^{2n-3}(\ba,\bb)dz\Big|\\
\nonumber &\leq 2\int_{\cC_t}\int_{\cD_{\cR(z)}(\ve)}  \frac{|\bw(z)|}{r^{1+\sigma}}dr d\cH^{2n-3}(\ba,\bb)dz\\
&\leq 2 C \ve^{\alpha-\sigma} \|\bw\|_{L^\infty(\Real^n)} \cH^1(\cC_t),
\end{align}
which shows that the family $(h_\ve(t)\ |\ \ve>0)$ is Cauchy. It then follows that the limit in \eqref{limhdef} defining $h(t)$ exists.

To see that the convergence of \eqref{varsLeneta} to $h(t)$ is uniform in $t$, it suffices to notice that the same calculation as above yields
\begin{align}
|\big(\text{Len}_{\sigma,\eta}(\cC_t,\Omega)\big)'-h(t)|&=|h_\eta(t)-h(t)|\\
&=\lim_{\ve\rightarrow 0} |h_\eta(t)-h_\ve(t)|\\
&\leq 2 C \eta^{\alpha-\sigma} \|\bw\|_{L^\infty(\Real^n)} \cH^1(\cC_t).
\end{align}
As $\bw$ is smooth with compact support and $\cC$ has finite length, $\cH^1(\cC_t)$ is bounded uniformly in $t\in\cI$. Thus, this shows the desired uniform convergence. From this we can conclude that
\beqn
\big(\text{Len}_{\sigma}(\cC_t,\Omega)\big)'=\lim_{\eta\rightarrow 0}\Big(\text{Len}_{\sigma,\eta}(\cC_t,\Omega)\Big)'=h(t).
\eeqn
Hence, for the variation of $\text{Len}_{\sigma}(\cC_t,\Omega)$ in the direction $\bw$ to vanish means that
\beqn
\lim_{\ve\rightarrow 0}\int_{\cC}\int_{\cU^2_\perp}\int_\ve^\infty\frac{\zeta(0,z,\ba,\bb,r) \bw(z)\cdot(\bb\wedge\ba)\bt(z)}{r^{1+\sigma}}dr d\cH^{2n-3}(\ba,\bb)dz=0.
\eeqn
As $\bw$ can be arbitrary in a neighborhood of $z_\circ$ we must have
\beqn
\lim_{\ve\rightarrow 0}\int_{\cU^2_\perp}\int_\ve^\infty\frac{\zeta(0,z_\circ,\ba,\bb,r) (\bb\wedge\ba)\bt(z_\circ)}{r^{1+\sigma}}dr d\cH^{2n-3}(\ba,\bb)=0.
\eeqn
Appealing to the definition of $\zeta$ in \eqref{defchi} and the fact that $z_\circ\in\cC$ was arbitrary, shows that \eqref{LenEL} holds for all $z\in\cC$.
\end{proof}

Since a curve connecting two points of minimal length has zero curvature, the preceding result motivates that we define the nonlocal curvature-vector $\bkappa_\sigma$ at $z\in\cC$ by
\beqn\label{kappas}
\bkappa_\sigma(z):=\lim_{\ve\rightarrow 0}\Big(\int_{\cA_\text{\rm e}^+(z,\ve)}-\int_{\cA_\text{\rm o}^+(z,\ve)}\Big) \frac{(\bb\wedge\ba)\bt(z)}{r^{1+\sigma}}d\cH^{2n-2}(\ba,\bb,r).
\eeqn
Notice that this vector is orthogonal to the curve at $z$, however there is no reason to believe that this vector is parallel to the classical normal to the curve.  The nonlocal scalar-curvature can be defined as the magnitude of this vector: $\kappa_\sigma(z):=|\bkappa_\sigma(z)|$.  

A comparison of \eqref{kappas} with the nonlocal mean-curvature \eqref{HsC} when $n=2$ is in order.  To do so, take the unit tangent of the curve $\bt$ and rotate it $90^\circ$ clockwise to obtain a normal vector $\bn$ to the curve.  The mean curvature of a two-dimensional curve is obtained from the mean-curvature vector by dotting it with $\bn$.  Thus, consider
\beqn\label{sH2d}
\bn(z)\cdot\bkappa_\sigma(z)=\Big(\int_{\cA_\text{\rm e}^+(z)}-\int_{\cA_\text{\rm o}^+(z)}\Big) \frac{(\ba\cdot\bt(z))(\bb\cdot\bn(z))-(\bb\cdot\bt(z))(\ba\cdot\bn(z))}{r^{1+\sigma}}d\cH^{2n-2}(\ba,\bb,r).
\eeqn
Focusing on the numerator of the integrand, notice that $\{\bt(z),\bn(z)\}$ and $\{\ba,\bb\}$ are both orthonormal basis for $\Real^2$ and, hence, there is a orthogonal transformation $\bR$ that takes $\{\bt(z),\bn(z)\}$ to $\{\ba,\bb\}$.  It follows that
\beqn
(\ba\cdot\bt(z))(\bb\cdot\bn(z))-(\bb\cdot\bt(z))(\ba\cdot\bn(z))=\text{det}(\bR).
\eeqn
Thus, this value is either $1$ or $-1$, depending on if these two basis have the same or opposite orientation.  Moreover, from the structure of orthogonal transformations in $\Real^2$, we know that
\beqn
\text{sgn}\det(\bR)=-\text{sgn}((\bb\cdot\bt(z))(\ba\cdot\bn(z)),
\eeqn
where $\text{sgn}\,x=\frac{x}{|x|}$ is the sign function.  Since in $\cA_\text{\rm e}^+(z,\ve)$ and $\cA_\text{\rm o}^+(z,\ve)$ we always have $\bb\cdot\bt(z)>0$, this implies that $\det(\bR)=1$ if and only if $\ba\cdot\bn(z)<0$ and $\det(\bR)=-1$ if and only if $\ba\cdot\bn(z)>0$.  Thus, if we define $\chi_\cC:\cC\times\cU^2\times\Real^+\rightarrow\Real$ by
\beqn
\chi_\cC(z,\ba,\bb,r):=\begin{cases}
1 & ((\ba,\bb,r)\in\cA^+_\text{e}(z,\ve)\ \text{and}\ \ba\cdot\bn(z)<0)\\
 & \qquad\text{or}\ ((\ba,\bb,r)\in\cA^+_\text{o}(z,\ve)\ \text{and}\ \ba\cdot\bn(z)>0)\\
-1 & ((\ba,\bb,r)\in\cA^+_\text{e}(z,\ve)\ \text{and}\ \ba\cdot\bn(z)>0)\\
 & \qquad\text{or}\ ((\ba,\bb,r)\in\cA^+_\text{o}(z,\ve)\ \text{and}\ \ba\cdot\bn(z)<0)\\
 0 & \text{otherwise},
\end{cases}
\eeqn
then \eqref{sH2d} can be written as
\beqn
\bn(z)\cdot\bkappa_\sigma(z)=\lim_{\ve\rightarrow \infty}\int_{(\cU^2_\perp\times (\ve,\infty))^+(z)} \frac{\chi_\cC(z,\ba,\bb,r)}{r^{1+\sigma}}d\cH^{2n-2}(\ba,\bb,r),
\eeqn
where
\beqn
(\cU^2_\perp\times(\ve,\infty))^+(z):=\{(\ba,\bb,r)\in\cU^2_\perp\times (\ve,\infty)\ |\ \bb\cdot\bt(z)>0\}.
\eeqn
Using the change of variables $(\ba,\bb,r)\mapsto z+2r\ba$ shows that $\bn(z)\cdot\bkappa_\sigma(z)$ agrees with $H_\sigma(z)$ in \eqref{HsC} with $n=2$ up to a multiplicative constant.  It is not surprising that they differ by a constant since $H_\sigma(z)$ in \eqref{HsC} is normalized to ensure that it converges in the appropriate limit to the classical mean curvature, but $\bkappa_\sigma(z)$ has not been properly normalized.  The investigation of what this normalization constant should be is left to future work.

\subsection*{Acknowledgments}
\noindent
I would like to thank Cornelia Mihaila for spotting an error in the published version of this manuscript and for comments on this corrected version.

\section{Appendix I: some change of variables}

Here we present several change of variable formulas that follow from the area formula that are needed in this work. Here we will employ the notation $\Real^+_0:=[0,\infty)$.

\begin{lemma}\label{lemCOV}
Consider the function 
$$
\Phi:\overline\cC\times\cU_\perp^2\times\Real^+_0\rightarrow \Real^n\times\Real^n
$$ 
defined by
\beqn\label{PhiCOV}
\Phi(z,\ba,\bb,\xi):=(z+\xi \ba,\bb)\quad \text{for all}\ (z,\ba,\bb,\xi)\in \overline\cC\times\cU_\perp^2\times\Real^+_0.
\eeqn
If $\cA$ is a subset of $\overline\cC\times\cU_\perp^2\times\Real^++0$ and $f:\Phi(\cA)\rightarrow\Real$ is an integrable function, then
\begin{multline}\label{COV}
\int_{\Phi(\cA)}\Big[ \sum_{(z,\ba,\bb,\xi)\in \Phi^{-1}(p,\bu)} f(p,\bu)\Big]d\cH^{2n-1}(p,\bu) \\
= \int_\cA f(z+\xi \ba,\bb)2^{-1/2}\xi^{n-2} |\bb\cdot\bt(z)| d\cH^{2n-1}(z,\ba,\bb,\xi),
\end{multline}
where $\bt(z)$ is a unit-vector tangent to the curve $\overline\cC$ at the point $z$.
\end{lemma}

\begin{proof}
It suffices to prove the result for a set of the form $\cA=\cC_\cA\times \cU_\cA\times\cI_\cA$, where $\cC_\cA\subseteq \overline\cC$, $\cU_\cA\subseteq \cU_\perp^2$, and $\cI_\cA\subseteq \Real^+_0$.  Moreover, by employing a partition of unity, we can reduce the problem to the case where $\cU_\cA$ is covered by one chart.  Thus, there is a set $U_\cA\subseteq \Real^{2n-3}$ and a diffeomorphism $\bchi:U_{\cA}\rightarrow \cU_\cA$.  Since $\cU_\cA\subseteq\Real^n\times\Real^n$, we can view this function as $\bchi=(\bchi_1,\bchi_2)$, where $\bchi_1,\bchi_2:\Real^{2n-3}\rightarrow\Real^n$.  Recall that if $g$ is an integrable function defined on $\cU_\cA$, then
\beqn\label{intdefU}
\int_{\cU_\cA} g(\ba,\bb)\, d\cH^{2n-3}(\ba,\bb)=\int_{U_\cA} g(\bchi_1(w),\bchi_2(w)) J_{\bchi}(w)\, dw,
\eeqn
where $J_{\bchi}=\sqrt{\text{det} (\nabla \bchi ^\top\nabla\bchi)}$ is the Jacobian of $\bchi$.  Also, if $h$ is an integrable function defined on $\cC_\cA$ and $\phi$ is a parameterization of $\bar\cC$, then
\beqn\label{intdefC}
\int_{\cC_\cA} h(z)\, dz=\int_{C_\cA} h(\phi(s)) |\phi'(s)|\, ds,
\eeqn
where $C_\cA$ is the subset of $\Real$ such that $\phi(C_\cA)=\cC_{\cA}$.

Set $A:=C_\cA\times U_\cA\times\cI_\cA$ and define $\tilde \Phi:A\rightarrow \Phi(\cA)$ by
\beqn
\tilde \Phi(s,w,\xi):=\Phi(\phi(s) ,\bchi_1(w),\bchi_2(w),\xi)=(\phi(s)+\xi \bchi_1(w), \bchi_2(w))
\eeqn
and $\tilde f:A\rightarrow \Real$ by
\beqn
\tilde f(s,w,\xi)=f(\phi(s)+\xi\bchi_1(w),\bchi_2(w)).
\eeqn
By the area formula (see Theorem 2.71 of Ambrosio, Fusco, and Pallara \cite{AFP}), it follows that
\beqn
\int_{\tilde\Phi(A)}\Big[ \sum_{(s,w,\xi)\in \Phi^{-1}(p,\bu)} \tilde f(s,w,\xi)\Big] d\cH^{2n-1}(p,\bu)=\int_A \tilde f(s,w,\xi)J_{\tilde \Phi}(s,w,\xi) d\cH^{2n-1}(s,w,\xi).
\eeqn
Thus, by \eqref{intdefU} and \eqref{intdefC}, the result will follow once it is shown that
\beqn\label{detneeded}
J_{\tilde \Phi}=2^{-1/2}\xi^{n-2}|\bchi_2\cdot \bt | |\phi'| J_{\bchi}.
\eeqn
To establish this, first notice that 
\beqn
\nabla\tilde\Phi=\begin{blockarray}{cccc}
1 & 2n-3 & 1  & \\
    \begin{block}{(c|c|c)c}
      \phi' & \xi \nabla \bchi_1  & \bchi_1   & n  \\ 
      \cline{1-3}
      \textbf{0}  & \nabla \bchi_2 & \textbf{0} & n  \\ 
    \end{block}
\end{blockarray}.
\eeqn
and hence, noting that $|\bchi_1|^2=1$ and $\nabla\bchi_1^\top\bchi_1=\bzero$, we have
\beqn
\nabla\tilde\Phi^\top\nabla\tilde\Phi=\begin{blockarray}{ccc}
    \begin{block}{(c|c|c)}
      |\phi'|^2  & \xi \phi'^\top\nabla\bchi_1& \bchi_1\cdot \phi' \Tstrut  \\ 
       \cline{1-3}
      \xi\nabla\bchi_1^\top \phi'  & \xi^2\nabla\bchi_1^\top\nabla\bchi_1+\nabla\bchi_2^\top\nabla\bchi_2 & 0  \Tstrut  \\
      \cline{1-3}
      \bchi_1\cdot \phi'  & 0 & 1  \Tstrut  \\
    \end{block}
\end{blockarray}\, .
\eeqn
Since switching rows or columns of a matrix does not change the absolute value of its determinant, we have 
\beqn
|\text{det}(\nabla\tilde\Phi^\top\nabla\tilde\Phi)| = \left |\text{det}\,
\begin{blockarray}{ccc}
    \begin{block}{(c|c|c)}
      |\phi'|^2 & \bchi_1\cdot \phi' & \xi \phi'^\top\nabla\bchi_1 \Tstrut  \\ 
      \cline{1-3}
      \bchi_1\cdot \phi'  & 1 & 0  \Tstrut  \\
      \cline{1-3}
      \xi\nabla\bchi_1^\top \phi'  & 0 & \xi^2\nabla\bchi_1^\top\nabla\bchi_1+\nabla\bchi_2^\top\nabla\bchi_2  \Tstrut  \\
    \end{block}
\end{blockarray}\, \right| .
\eeqn

Recall that the determinant of a block matrix can be computed using
\begin{align*}
\det\begin{blockarray}{cc}
  &  \\
    \begin{block}{(c|c)}
      \bA & \bB   \\ 
      \cline{1-2}
      \bC & \bD \Tstrut \\ 
       \end{block}
\end{blockarray}&=\det(\bA)\det(\bD-\bC\bA^{-1}\bB).
\end{align*}
This identity will be used with $\bD=\xi^2\nabla\bchi_1^\top\nabla\bchi_1+\nabla\bchi_2^\top\nabla\bchi_2$.  Let $P_{\bchi_1}$ denote the projection onto the plane orthogonal to the vector $\bchi_1$.  After some computation, using the identity $|P_{\bchi_1} \phi'|^2=|\phi'|^2-(\bchi_1\cdot\phi')^2$, one finds that 
\beqn
|\text{det}(\nabla\tilde\Phi^\top\nabla\tilde\Phi)|=|P_{\bchi_1}\phi'|^2 \left |\text{det}\Big[ \nabla\bchi_1^\top(\xi^2\textbf{1}_n-\frac{\xi^2}{|P_{\bchi_1}\phi'|^2} P_{\bchi_1}\phi'\otimes P_{\bchi_1}\phi')\nabla\bchi_1+\nabla\bchi_2^\top\nabla\bchi_2\Big]\right |.
\eeqn
where $\textbf{1}_n$ is the identity function on $\Real^n$.  It follows that 
\beqn\label{det2}
|\text{det}(\nabla\tilde\Phi^\top\nabla\tilde\Phi)|=|P_{\bchi_1}\phi'|^2 \Big | \text{det}\Big [ \nabla\bchi^\top
\begin{blockarray}{cc}
  &  \\
    \begin{block}{(c|c)}
      \xi^2\textbf{1}_n-\xi^2 \tilde\bt\otimes \tilde\bt & \textbf{0}   \\ 
      \cline{1-2}
      \textbf{0} & \textbf{1}_n \Tstrut \\ 
       \end{block}
\end{blockarray}\nabla\bchi \Big]\Big |,
\eeqn
where $\tilde\bt=P_{\bchi_1}\phi'/|P_{\bchi_1}\phi'|$.  To simplify the right-hand side of the previous equation, set $\bL$ equal to the square $2n\times 2n$ matrix in the previous equation between $\nabla \bchi^\top$ and $\nabla\bchi$ and recall the fact that
\beqn\label{detfact}
\text{det}(\nabla\bchi^\top\bL\nabla\bchi)=J_{\bchi}^2\text{det}(\bI^\top\bL\bI),
\eeqn
where $\bI$ is the natural injection of the range of $\nabla\bchi$ into $\Real^{2n}$.  One can find an orthonormal basis for $\Real^n$ of the form
$$(\be_1,\be_2,\dots,\be_{n-2},\bchi_1,\bchi_2)$$
such that $\tilde\bt$ can be written as a linear combination of $\be_{n-2}$ and $\bchi_2$.  Notice that
\beqn\label{U2basis}
\{(\be_1,\bzero),\dots, (\be_{n-2},\bzero),(\bzero,\be_1),\dots,(\bzero,\be_{n-2}),\frac{1}{\sqrt{2}}(\bchi_2,-\bchi_1)\}
\eeqn
is an orthonormal basis for the tangent space of $\cU_\perp^2$ at $(\bchi_1,\bchi_2)$.  By writing the matrix
$$\bI^\top\begin{blockarray}{cc}
  &  \\
    \begin{block}{(c|c)}
      \xi^2\textbf{1}_n-\xi^2 \tilde\bt\otimes \tilde\bt & \textbf{0}   \\ 
      \cline{1-2}
      \textbf{0} & \textbf{1}_n \Tstrut \\ 
       \end{block}
\end{blockarray}\,\bI$$
relative to the basis \eqref{U2basis}, one can compute its determinant and, hence, putting together \eqref{det2} and \eqref{detfact}, we have that
\beqn
|\text{det}(\nabla\bchi^\top\bL\nabla\bchi)|=2^{-1/2}\xi^{2n-4}|P_{\bchi_1}\phi'|^2 |\bchi_2\cdot \tilde\bt|^2J^2_{\bchi}.
\eeqn
Since $|P_{\bchi_1}\phi'|^2 |\bchi_2\cdot \tilde\bt|^2=|\bchi_2 \cdot  P_{\bchi_1}\bt|^2|\phi'|^2=|\bchi_2\cdot\bt|^2|\phi'|^2$, this proves \eqref{detneeded}.
\end{proof}

\begin{figure}[h]
\centering
\includegraphics[width=4in]{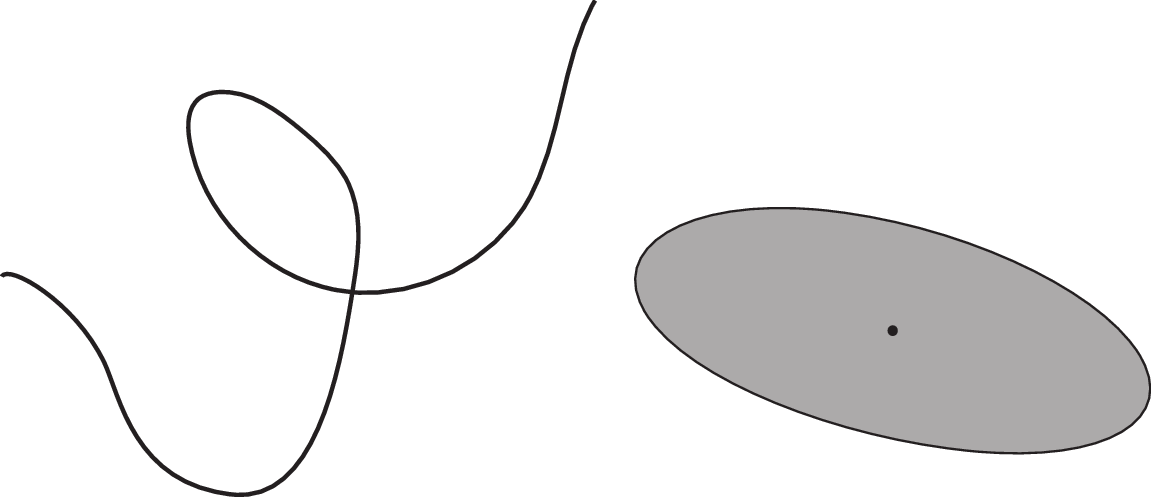}
\thicklines
\put(-177,68){$z$}
\put(-179,54){$\bullet$}
\put(-178,55){\rotatebox[origin=c]{-8}{$\vector(1,0){25}$}}
\put(-165,45){$\ba$}
\put(-90,33){$z+\xi\ba$}
\put(-74.5,50){\rotatebox[origin=c]{70}{$\vector(1,0){25}$}}
\put(-73,55){$\bb$}
\put(-250,50){$\cC$}
\caption{A depiction of the disc associated with $\Xi(z,\ba,\bb,\xi,r)$.}
\label{figcovXi}
\end{figure}

A slight variation on the change of variables formula \eqref{COV} will be needed.  Namely, we will require a change of variables associated with the function
$$
\Xi:\overline\cC\times\cU_\perp^2\times\Real^+_0\times\Real^+\rightarrow \Real^n\times\Real^n\times\Real
$$ 
defined by
\beqn\label{XiCOV}
\Xi(z,\ba,\bb,\xi,r):=(z+\xi \ba,\bb,r)\quad \text{for all}\ (z,\ba,\bb,\xi,r)\in \overline\cC\times\cU_\perp^2\times\Real^+_0\times\Real^+.
\eeqn
The function $\Xi$ allows us to describe discs using points on $\bar\cC$; see Figure~\ref{figcovXi}.
Since $\Xi$ acts like the identity on the last variable $r$, the previous lemma immediately implies that if $\cA$ is a subset of $\overline\cC\times\cU_\perp^2\times\Real^+_0\times\Real^+$ and $f:\Xi(\cA)\rightarrow\Real$ is an integrable function, then
\begin{multline}\label{XiCOV2}
\int_{\Xi(\cA)}\Big[ \sum_{(z,\ba,\bb,\xi,r)\in \Xi^{-1}(p,\bu,r)} f(p,\bu,r)\Big]d\cH^{2n}(p,\bu, r) \\
= \int_\cA f(z+\xi \ba,\bb,r)2^{-1/2}\xi^{n-2} |\bb\cdot \bt(z)|  d\cH^{2n}(z, \ba,\bb,\xi,r).
\end{multline}

\begin{figure}[h]
\centering
\includegraphics[width=4in]{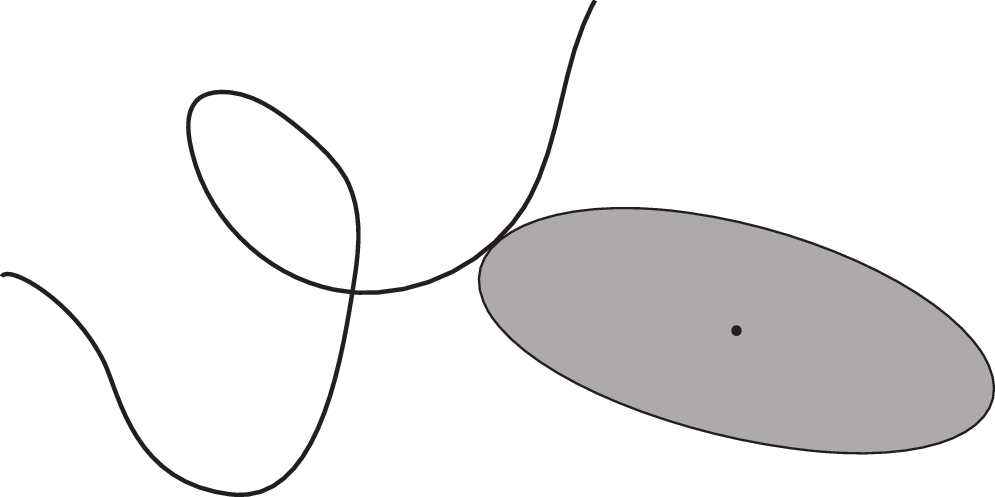}
\thicklines
\put(-149,82){$z$}
\put(-149,71){$\bullet$}
\put(-147,69){\rotatebox[origin=c]{-20}{$\vector(1,0){25}$}}
\put(-144,63){$\ba$}
\put(-90,40){$z+r\ba$}
\put(-85,56){\rotatebox[origin=c]{70}{$\vector(1,0){25}$}}
\put(-83,65){$\bb$}
\put(-250,50){$\cC$}
\caption{A depiction of the disc associated with $\Psi(z,\ba,\bb,r)$.}
\label{figcovPhi}
\end{figure}

Another useful change of variables will be needed.  This time the function, which is closely related to $\Xi$, will allow us to describe all discs whose boundary touches the curve $\overline\cC$; see Figure \ref{figcovPhi}.  Define the function
$$
\Psi:\overline\cC\times\cU_\perp^2\times\Real^+\rightarrow \Real^n\times\Real^n\times\Real
$$ 
by
\beqn\label{PsiCOV}
\Psi(z,\ba,\bb,r):=(z+r\ba,\bb,r)\quad \text{for all}\ (z,\ba,\bb,r)\in \overline\cC\times\cU_\perp^2\times\Real^+.
\eeqn
It is possible to show that if $\cA$ is a subset of $\overline\cC\times\cU_\perp^2\times\Real^+$ and $f:\Psi(\cA)\rightarrow\Real$ is an integrable function, then
\begin{multline}\label{PsiCOV2}
\int_{\Psi(\cA)}\Big[ \sum_{(z,\ba,\bb,r)\in \Psi^{-1}(p,\bu,r)} f(p,\bu,r)\Big]d\cH^{2n-1}(p,\bu, r) \\
= \int_\cA f(z+r\ba,\bb,r)2^{-1/2}r^{n-2} \sqrt{(1+r^2)(\ba\cdot\bt(z))^2+2(\bb\cdot\bt(z))^2}  d\cH^{2n-1}(z,\ba,\bb,r).
\end{multline}
The proof of this result is similar to that of Lemma~\ref{lemCOV} and, thus, will not be presented.  

\section{Appendix II: a transport theorem}



Seguin \cite{SeguinATT} established a formula for the rate of change of an integral in which both the domain of integration and the integrand depended on time under certain conditions on the domain and integrand. The goal of this appendix is to establish a version of this result, using some of the results in Seguin \cite{SeguinATT}, that can be applied to compute the first variation of the $s$-length.

\begin{theorem}\label{thmATTmod}
Let $\cM$ be a $d$-dimensional submanifold of $\Real^N$. For each $t\in\cI$, $\cI$ being an open interval of $\Real$, let $\cO_t$ be an open subset of $\cM$ that is locally of finite perimeter with exterior unit normal $\bnu(t,\cdot)$ such that there exists a $(d-1)$-dimensional Riemannian manifold $\cN$ and a function $\Theta\in C^1(\cI\times\cN,\Real^d)$ such that for all $t\in\Real$ the following conditions hold:
\begin{enumerate}[label=(\textit{D\arabic*})]
\item \label{D1} the gradient of $\Theta_t:=\Theta(t,\cdot):\cN\rightarrow\Real^n$ is injective $\cH^{d-1}$-a.e.,
\item \label{D2} $\partial^*\cO_t$ and $\Theta_t(\cN)$ differ by a set of $\cH^{d-1}$-measure zero, and
\item \label{D3} $\cH^{d-1}(\{ x\in\partial^*\cO_t\ |\ \cH^0(\Theta_t^{-1}(\{x\}))>1\})=0$.
\end{enumerate}
The `velocity' $\bv$ associated with $\Theta$ is defined by
\beqn\label{veldef}
\bv(t,x):=\Theta'(t,\Theta_t^{-1}(x))\quad \text{for $\cH^{d-1}$-a.e.}\ x\in \partial^*\cO_t,
\eeqn
where the prime denotes the partial derivative with respect to $t$. Set
\beqn\label{defcF}
\cF:=\{(t,x)\in\cI\times\cM\ |\ x\in \cO_t\}\subseteq \cI\times\cM=:\cE,
\eeqn
and consider $f:\cE\rightarrow\Real$ such that
\begin{enumerate}[label=(\textit{F\arabic*})]
\item \label{F1} $f\in L^\infty(\cF)$ and $f(t,\cdot)\in L^1(\cO_t)$ for all $t\in\cI$,
\item \label{F2} $\cH^d$-a.e.~$(x,t)\in\cE$ is a Lebesgue point of $f$,
\item \label{F3} for any bounded interval $\cJ\subseteq\cI$, $\int_\cJ\int_{\partial^*\cO_s}|f\bv\cdot\bnu|\, d\cH^{d-1}<\infty$,
\item \label{F4} the function $t\mapsto \int_{\partial^*\cO_t}f\bv\cdot\bnu\, d\cH^{d-1}$ is continuous on $\cI$, and
\item \label{F5} $f'=0$.
\end{enumerate}
It follows that
\beqn\label{eqATT}
\Big(\int_{\cO_t} f\, d\cH^d \Big)'= \int_{\partial^*\cO_t} f \bv\cdot\bn\, d\cH^{d-1}.
\eeqn
\end{theorem}

\begin{proof}
Fix $t_\circ,t\in\Real$, and consider the set
\beqn
\cW:=\{(s,x)\in\Real\times\cM\ |\ t_\circ<s<t,\ x\in \cO_s\}\subseteq \Real\times\cM.
\eeqn
The arguments in Seguin \cite{SeguinATT} show that $\cW$ is locally a set of finite perimeter with reduced boundary given by $\partial^*\cW=\cB\cup\cS\cup\cT$, where
\beqn\label{rbdyWparts}
\cB:=\{t_\circ\}\times \cO_{t_\circ},\quad \cS:=\bigcup_{s\in (t_\circ,t)}\{s\}\times\partial^*\cO_s,\quad \text{and}\quad \cT:=\{t\}\times\cO_t.
\eeqn
Using the purely time-like vector $\btau:=(1,\textbf{0})$, the exterior unit-normal $\bn$ to $\cW$ is given by
\beqn\label{bnuformula}
\bn=\begin{cases}
-\btau & \text{on}\ \cB,\\
\frac{\bnu-(\bv\cdot \bnu)\btau}{\sqrt{1+(\bv\cdot\bnu)^2}} & \text{on}\ \cS,\\
\btau & \text{on}\ \cT,
\end{cases}
\eeqn
where $\bv$ is defined in \eqref{veldef}. It was also established in Seguin~\cite{SeguinATT} that if $\phi$ is an integrable function on $\cS$, then
\beqn\label{cScoAF}
\int_\cS\phi\, d\cH^d=\int_{t_\circ}^t\int_{\partial^*\cO_s}\phi\sqrt{1+(\bv\cdot\bnu)^2}\, d\cH^{d-1}ds.
\eeqn

Let $\nabla_\cE$ and $\text{div}_\cE$ denote the (covariant) gradient and divergence of a function defined on a subset of $\cE$, and, as before, a prime will denote a partial derivative with respect to $t$. Notice that for any $\psi\in C^\infty_c(\cE)$ we have
\beqn
\int_\cE f\btau\cdot\nabla_\cE\psi\, d\cH^{d+1}=\int_\cE f \psi' \,d\cH^{d+1}=-\int_\cE f' \psi \,d\cH^{d+1}=0.
\eeqn
It follows that $\text{div}_\cE(f\btau)=0$ in the weak sense. 

For each $n\in\Nat$, let $\phi_n:\cE\rightarrow\Real$ be a Lipschitz function with compact support such that $\phi_n'=0$ on $\cW$, $|\phi_n|\leq 1$, and $\phi_n=1$ on a set $\cA_n$ with $\cup_{n\in\Nat}\cA_n=\cE$. As $f\btau$ is bounded and has zero divergence, the divergence theorem\footnote{For the divergence theorem in the case of bounded divergence measure fields involving sets of finite perimeter see \u{S}ilhav\'y \cite{S05} or Chen and Torres \cite{CT05}.}  implies that
\beqn
\int_{\partial^*\cW} \phi_n q\, d\cH^d = \int_\cW \nabla_\cE\phi_n \cdot (f\btau)\, d\cH^{d+1}=\int_\cW \phi_n' f \, d\cH^{d+1}=0,
\eeqn
where $q$ is the normal trace of $f\btau$\footnote{The concept of normal trace in this context was first introduced by Anzellotti \cite{A83}}. Recall that $\partial^*\cW$ is the union of the sets defined in \eqref{rbdyWparts}. Since $\cH^d$-a.e.~point of $\cE$ is a Lebesgue point of $f$, the normal trace $q=f\btau\cdot\bn$ $\cH^d$-a.e.~on $\partial^*\cW$ (see \u{S}ilhav\'y \cite[Remark 4.3]{S05}.) Using this formula for the normal trace, \eqref{bnuformula}, and \eqref{cScoAF}, the previous equation can be written as
\beqn\label{NDTT}
-\int_{\cO_{t_\circ}}\phi_n f\, d\cH^d+\int_{\cO_t}\phi_n f\, d\cH^d-\int_{t_\circ}^t\int_{\partial^*\cO_s}\phi_n f\bv\cdot\bnu\, d\cH^{d-1}ds=0.
\eeqn
Taking the limit as $n$ goes to infinity, the assumed integrability conditions on $f$ imply that the limit of each of the integrals in the previous equation exist and
\beqn\label{NDTT2}
-\int_{\cO_{t_\circ}}f\, d\cH^d+\int_{\cO_t}f\, d\cH^d-\int_{t_\circ}^t\int_{\partial^*\cO_s}f\bv\cdot\bnu\, d\cH^{d-1}ds=0.
\eeqn
As $\partial^*\cO_s=\Theta_s(\cN)$ up to a set of $\cH^{d-1}$-measure zero, it follows that
\beqn
\int_{\partial^*\cO_s}f\bv\cdot\bnu\, d\cH^{d-1}=\int_{\Theta_s(\cN)}f\bv\cdot\bnu\, d\cH^{d-1},\quad s\in\cI.
\eeqn
From \ref{F4} the function $s\mapsto \int_{\Theta_s(\cN)}f\bv\cdot\bnu\, d\cH^{d-1}$ is continuous. Thus, the third term on the left-hand side of \eqref{NDTT2} is differentiable with respect to $t$ and differentiating \eqref{NDTT2} yields the desired result.
\end{proof}

\bibliographystyle{acm}
\bibliography{nonlength} 

\end{document}